\documentclass[a4paper,11pt]{amsart} 
\usepackage{amscd,amssymb}
\usepackage{amsfonts}
\usepackage{graphicx,epsf}
\usepackage{enumerate,color}
\usepackage[mathcal]{euscript}
\usepackage{hyperref}
\numberwithin{equation}{section}
\usepackage{epstopdf}
\usepackage{epsfig}

\newtheorem{theorem}{Theorem}[section]
\newtheorem{lemma}[theorem]{Lemma}

\newtheorem{corollary}[theorem]{Corollary}

\newcommand{\R}{\mathbb{R}}
\newcommand{\C}{\mathbb{C}}
\newcommand{\Z}{\mathbb{Z}} 
\newcommand{\N}{\mathbb{N}}

\newcommand{\supp}{\mathrm{supp}}
\newcommand{\ND}{\mathcal{R}}

\newcommand{\NDmat}{\mathbf{R}}

\newcommand{\pNDdiffMat}{{\mathbf{R}}^\Gamma_{\sigma,1}}
\newcommand{\NDdiffMat}{\mathbf{R}_{\sigma,1}}

\newcommand{\pNDdiff}{{\mathcal{R}}^\Gamma_{\sigma,1}}
\newcommand{\NDdiff}{\mathcal{R}_{\sigma,1}}
\newcommand{\Tin}{T_\diamond(\bndry)}
\newcommand{\Tout}{T(\bndry)}

\DeclareMathOperator{\parphi}{\psi}
\newcommand{\Hmhalf}{L^2_\diamond(\bndry)}
\newcommand{\Hhalf}{L^2_\diamond(\bndry)}
\newcommand{\parspace}{L^2_\Gamma(\bndry)}
\DeclareMathOperator{\parI}{\mathcal{I}}

\DeclareMathOperator{\dbar}{\overline{\partial}}

\DeclareMathOperator{\dnu}{\partial_\nu}

\DeclareMathOperator{\bndry}{ {\partial\Omega} }
\DeclareMathOperator{\T}{\mathbf{t}}

\begin{document}

\title[Approximation of full-boundary data]{Approximation of full-boundary data from partial-boundary electrode measurements}
\author[A. Hauptmann]{A. Hauptmann }\thanks{Department of Computer Science, University College London, United Kingdom.}
\date{\today}

\begin{abstract}
\noindent  
Measurements on a subset of the boundary are common in electrical impedance tomography, especially any electrode model can be interpreted as a partial-boundary problem. The information obtained is different to full-boundary measurements as modelled by the ideal continuum model. In this study we discuss an approach to approximate full-boundary data from partial-boundary measurements that is based on the knowledge of the involved projections. The approximate full-boundary data can then be obtained as the solution of a suitable optimization problem on the coefficients of the Neumann-to-Dirichlet map. By this procedure we are able to improve the reconstruction quality of continuum model based algorithms, in particular we present the effectiveness with a D-bar method. Reconstructions are presented for noisy simulated and real measurement data.
\end{abstract}
\maketitle


\section{Introduction}
Electrical Impedance Tomography (EIT) is an emerging noninvasive imaging technique, that seeks to gain knowledge about the conductivity of an object by injecting currents at the boundary and measuring the resulting voltage distribution. In many applications it might be that we cannot access the full boundary and hence electrodes can only be placed on a subset. In this case the data error depends linearly on the length of the missing domain, as shown for the the continuum setting in \cite{Hauptmann2017} and for a realistic electrode setting in \cite{Hyvoenen2009}. In the following we show how to overcome this restriction by utilizing our knowledge of the involved operators and formulate a numerical procedure to approximate full-boundary data from potentially noisy electrode measurements.

In this study we consider a bounded domain $\Omega\subset\R^2$ with smooth boundary $\partial\Omega$ and the unknown conductivity $\sigma\in L^\infty(\Omega)$ is assumed to be strictly positive. 
 A mean free current $\psi$ is injected on a (possibly not connected) subset $\Gamma\subset\bndry$. 
The problem of EIT can be modelled by the conductivity equation with Neumann boundary condition,
\begin{equation}\label{eqn:introCondEq}
\begin{array}{rl}
\nabla\cdot\sigma\nabla u = 0, &\mbox{ in }\Omega,\\
\sigma \dnu u = \psi, &\mbox{ on } \Gamma \subset\bndry \\
\sigma \dnu u = 0, &\mbox{ on } \Gamma^c=\partial\Omega\backslash\Gamma.
\end{array}
\end{equation}
Here $\nu$ denotes the outward normal, for uniqueness we require $\int_{\bndry}u\;ds=0$ \newline and for well-posedness $\int_{\Gamma}\psi\;ds=0$. Ideally we want to measure the Neumann-to-Dirichlet map (ND map), also denoted as current-to-voltage map, that maps every possible current pattern to the corresponding voltage distribution on the boundary,
\begin{equation}\label{eqn:introNDmap}
\ND_\sigma:L_\diamond^{2}(\bndry) \to L_\diamond^{2}(\bndry),  \hspace{0.25 cm} \left.\ND_\sigma\psi=u\right|_{\bndry},
\end{equation}
where $L_\diamond^{2}(\bndry)$ denotes the space of mean free $L^2$-functions on the boundary.
In case of just partially supported currents we can not measure the full ND map, instead we obtain knowledge of a partial ND map. It has been proposed in \cite{Hauptmann2017}, that this partial ND map can be represented as a composition of the actual ND map and a partial-boundary map, that maps currents supported on the full boundary to the partial boundary. In this study we will utilize the knowledge of the involved mappings to formulate an optimization problem, by which we can compute an approximation to the (full-boundary) ND map.

Electrical impedance tomography is a highly nonlinear inverse problem, hence the reconstruction task is especially challenging and can benefit from improved data. Promising applications of EIT include pulmonary imaging \cite{Cinnella2015, Grant2011, Karagiannidis2015, Reinius2015}, along with nondestructive testing and evaluation of materials in industrial and engineering applications \cite{Hallaji2014, Hou2009, Karhunen2010, Kaup1995}. The partial-data problem is especially relevant for patient monitoring, since faulty or displaced electrodes can lead to severe corruption of data. Reconstructions from partial-boundary data typically suffer from geometrical distortion. Furthermore, less electrodes available for data acquisition lead to a loss of linearly independent basis functions for the representation of the measured ND map. Formulating an optimization problem on the coefficients of the ND map is a novel approach to obtain information of the full-boundary problem. This approximation process is capable of recovering important features lost in the ND map as well as extending the number of linearly independent basis functions. A related approach has been proposed in \cite{Harrach2015a}, where voltage data is recovered from missing or faulty electrodes by an interpolation on the sensitivity matrix.

Due to the aforementioned instabilities, carefully designed regularization strategies are needed. D-bar methods, for example, offer a direct and robust reconstruction procedure. In this study we use a D-bar method that has been specifically formulated for ND maps \cite{Hauptmann2017} to evaluate the improved data. This algorithm is based on a proven regularization strategy for full-boundary Dirichlet-to-Neumann data \cite{Knudsen2009}.  The biggest advantage of D-bar methods lies in their direct nature and hence they do not get stuck in local minima, they are robust to modelling errors \cite{Murphy2009}, are capable of real time imaging \cite{Dodd2014}, and recently allow the incorporation of prior information \cite{Alsaker2016a}. Other theoretically based methods that operate on (or can be adjusted to) ND maps, and hence are related to this study, include the monotonicity method \cite{Harrach2013,Harrach2015,Tamburrino2002}, the factorization method \cite{Kirsch2005,Lechleiter2006,Lechleiter2008a}, and the enclosure method \cite{Ikehata2000c,Ikehata2000a,Ikehata2004}. In particular additional regularization can be incorporated as discussed in \cite{Garde2016,Harrach2016a}.

Further approaches to deal with partial-boundary data are based on iterative algorithms that do not operate on ND maps, but rather use a forward map for EIT that maps the conductivity to voltage measurements on the boundary. These methods are typically very powerful due to flexibility in incorporating \emph{a priori} information, but tend to be sensitive to modelling errors in the forward solver. Methods addressing partial data include domain truncation approaches \cite{Calvetti2015,Calvetti2015a}, region of interest imaging \cite{Liu2015,Liu2015a}, and electrode configuratons covering parts of the boundary \cite{Mueller1999,Vauhkonen1999a}.

This paper is organized as follows. In Section \ref{sec:Models} we introduce an \emph{electrode continuum model} (ECM) that involves the partial ND map and is more suitable to interpret measurements from the well established \emph{complete electrode model} (CEM) \cite{Cheng1989,Somersalo1992} in a continuum setting. In Theorem \ref{theo:approx} we show that the approximation error of the ECM to the CEM for partial-boundary measurements depends linearly on the length of the electrodes, but is independent of the missing boundary. In particular that means the approximation property of the ECM to electrode measurements stays stable if the electrode size does not change, while the amount of covered boundary can decrease.

In Section \ref{sec:OptProb} we use the ECM to derive an optimization problem to obtain an approximate ND map, representing full-boundary data, from partial-boundary measurements. We discuss how to approximate continuum data from electrode measurements and state the proposed algorithm.
An auxiliary result shows that the ND map for rotationally symmetric conductivities can be obtained from the measurement of a single current pattern in the continuum setting. In Section \ref{sec:CompResults} we discuss computational aspects to approximate the ND maps and present reconstructions for simulated noisy partial-boundary data from the ECM and CEM. Additionally, we present the effectiveness for real measurements with a pairwise injection current pattern from the KIT4 system located in Kuopio, University of Eastern Finland. A short discussion is presented subsequently. Section \ref{sec:conclusion} presents the conclusions to this study.

\section{Continuum and electrode models}\label{sec:Models}
For the following analysis it is important to discuss the difference of continuum based models and the \emph{complete electrode model} for EIT. Most of the mathematical analysis is based on the continuum model, given by \eqref{eqn:introCondEq}, with a connected domain for current injection. The ideal measurement (optimally from full boundary) is given by the ND map $\ND_\sigma$ and is used in many theoretically based reconstruction algorithms.

Let us now introduce the basic setting for a realistic electrode set-up. Let the partial boundary $\Gamma\subset\bndry$ consist of a union of mutually disjoint electrodes $E_m,\ m=1,\dots,M$, with their characteristic functions $\chi_m$. Each electrode is modelled as a connected and open subset of the boundary $\bndry$. Hence we have $\Gamma=\cup_{m}E_m$, which serves as the domain of current input and voltage measurement. We assume throughout the paper that all electrodes have the same size, that is $|E_m|=|E|>0$ for all $m=1,\dots,M.$

Typically one assigns to each electrode a value for the current injection, similarly the measurement returns only one voltage value. Following \cite{Hyvoenen2009}, we define a space of piecewise constant functions for the electrode data by
\begin{equation}\label{eqn:piecewiseSpace}
T(\bndry):=\left\lbrace V\in L^2(\bndry) \left| \ V=\sum_{m=1}^M \chi_m V_m,\ V_m\in\R \text{, and } V=0 \text{ on } \Gamma^c \right. \right\rbrace.
\end{equation}
It follows clearly that $T(\bndry)\subset L^2(\bndry)$ and we denote $T_\diamond(\bndry)=T(\bndry)\cap L^2_\diamond(\bndry)$.

Real measurement data is typically modelled by the \emph{complete electrode model} (CEM). The first difference to \eqref{eqn:introCondEq} is that the input current is modelled to equal the integral over the current density on each electrode. Furthermore, the CEM takes contact impedances into account, that model the electrochemical effect of a thin resistive layer forming between the electrodes and the measured target, which is incorporated by a Robin boundary condition. The full model is then given for the input current $J\in \Tin$ and measured voltages $U\in \Tout$ by
\begin{equation}
\label{eqn:CEM}
\left\{
\begin{array}{rcl}
\nabla\cdot\sigma\nabla u &=& 0 \mbox{ in }\Omega,\\
\sigma \dnu u &=& 0 \mbox{ on }\partial\Omega\backslash {\Gamma} \\
u+z \sigma \dnu u &=& U \mbox{ on }\Gamma \\
\frac{1}{|E_m|}\int_{E_m} \sigma \dnu u &=& J_m \mbox{ for } 1\leq m \leq M.
\end{array}
\right.
\end{equation}
Here we assume a constant contact impedance $z>0$. Further, we assume that the potentials are mean free, i.e. $\int_{\bndry}u\; ds=0$. The corresponding measurement operator for the CEM is given by
\[
\ND^{E}_\sigma:J\mapsto U, \hspace{0.5cm}  \Tin \to \Tout.
\]
This operator is linear and continuous as discussed in \cite{Hyvoenen2004,Hyvoenen2009}.

\subsection{Modelling continuum partial-boundary measurements}\label{sec:ModelCont}
The construction presented here follows the previous study in \cite{Alsaker2017,Hauptmann2017}. We consider a two dimensional bounded domain $\Omega\subset\R^2$, a current with zero mean on the partial boundary $\Gamma\subset\bndry$ is injected. Since the currents can be only injected on a subset, we use a subspace of functions supported on $\Gamma$ denoted as
\[
\parspace:=\left\{\varphi\in \Hhalf \left|\ \supp( \varphi)\subset\Gamma \text{ and } \int_\Gamma \varphi = 0\right\}\right. .\]
We note that in the electrode setting we clearly have $\Tin\subset\parspace\subset\Hhalf$. Let now ${\psi}\in\parspace$, then we can state the conductivity equation \eqref{eqn:introCondEq} with Neumann boundary condition simplified as 
\begin{equation*}
\begin{array}{rl}
\nabla\cdot\sigma\nabla u = & 0, \quad \mbox{ in }\Omega,\\
\sigma \frac{\partial u}{\partial \nu}  =& \psi, \quad  \mbox{ on } \bndry. \\
\end{array}
\end{equation*}
The measurement is modelled by application of the ND map, that is given a current pattern $\psi\in\parspace$, we obtain
\[
\ND_\sigma \psi = u|_{\bndry}.
\]
The resulting voltages are supported on the whole boundary $\bndry$ and represent an ideal measurement. Let us assume for now that we can measure on the full-boundary $\bndry$. 
In order to understand partial-boundary measurements we introduce a linear and bounded operator $\parI:\Hmhalf\to\parspace$ and define the partial ND map by
\[
\ND^\Gamma_\sigma:=\ND_\sigma\parI. 
\]
A choice for the operator $\parI$ is introduced in the following Section \ref{sec:parBoundMap}, also choices in a purely continuum setting are discussed in \cite{Hauptmann2017}. For computational purposes one wants to represent the ND map with respect to an orthonormal basis $\varphi_n\in\Hhalf $ for $n\in\Z\backslash \{0\}$.
Let the partial-boundary function $\psi$ be produced by $\psi=\parI\varphi$ for some basis function $\varphi\in\Hmhalf$. Then we immediately obtain the identity of ND maps
\begin{equation}
\label{eq:centralIdent}
\ND_\sigma\psi={\ND}_\sigma\parI\varphi={\ND}^\Gamma_\sigma \varphi.
\end{equation}
Further, this means for the measurement 
\[
\ND_\sigma\psi=\ND^\Gamma_\sigma \varphi = u|_{\bndry}
\]
from which we see, that in fact we measure the partial ND map
\[
\ND^\Gamma_\sigma:\Hhalf\to\Hhalf,
\]
with respect to the basis function $\varphi$. This way we are able to represent a finite-dimensional matrix approximation to the ND map with respect to the chosen basis. We denote this matrix approximation of the {partial ND map} by $\mathbf{R}^\Gamma_\sigma$ and compute it from the measurements $\ND_\sigma\psi_n=u_n|_{\bndry}$ by
\begin{equation}\label{eqn:matrixApprox}
({\mathbf{R}}^\Gamma_\sigma)_{n,\ell}=({\ND}^\Gamma_\sigma \varphi_n,\varphi_\ell)=(\ND_\sigma\parphi_n,\varphi_\ell)=\int_{\bndry} u_n|_{\bndry}(s)\varphi_\ell(s) ds.
\end{equation}

The question we want to investigate in the following is, whether we can recover information of the (full-boundary) ND map $\ND_\sigma$ from the measurement of $\ND^\Gamma_\sigma$. From the identity \eqref{eq:centralIdent}, we notice that we have knowledge of $\varphi$ and $\psi$, as well as a measurement of $\ND^\Gamma_\sigma$. This leaves us with one unknown, the ND map $\ND_\sigma$. This identity will be the basis of the approximation process in Section \ref{sec:OptProb}.

It is easily seen by linearity that the central identity \eqref{eq:centralIdent} holds for the difference maps $\ND_{\sigma,1}:=\ND_\sigma-\ND_1$, resp. ${\ND}^\Gamma_{\sigma,1}:={\ND}^\Gamma_\sigma-{\ND}^\Gamma_1$. The difference map is inherent to most continuum based methods and hence we will work mostly with the difference map. It is important to mention that the reference does not need to be the unit conductivity, in fact it can be any smooth conductivity, see for more \cite{Mueller2012,Novikov2004}.

\subsection{A continuum model for electrode data}\label{sec:parBoundMap}
In this section we relate the partial-boundary problem in the continuum setting to the complete electrode model. The construction follows essentially the work in \cite{Hyvoenen2009}. The main differences are, that we concentrate on the partial-boundary problem with larger parts missing and establish an error estimate that depends on the length of the electrodes, rather than the length of the missing boundary as in the original work \cite{Hyvoenen2009}. In particular, we show that by adjusting the continuum model we can obtain a better error estimate to CEM measurements than considering the classical continuum setting, that means we have a better interpretation of CEM measurements for the optimization procedure discussed in Section 3.

First we need to utilize the concept of extended electrodes $\{\widetilde{E}_m\}_{m=1}^M$, that we define as open, connected and mutually disjoint subsets of $\bndry$, with the property that
\[
E_m\subset\overline{\widetilde{E}}_m \quad \text{and} \quad \bigcup_{m=1}^M \overline{\widetilde{E}}_m = \bndry, \quad \text{for } m=1,\dots,M.
\]
Now we can present the involved projections, for a suitable electrode input from continuum data we use the nonorthogonal projection, introduced in \cite{Hyvoenen2009}, given by
\begin{equation}\label{eqn:nonortho_proj}
Q: \varphi \mapsto \sum_{m=1}^M \frac{{\chi}_{_m}}{|E_m|}\int_{\widetilde{E}_m} \varphi\ ds, \quad L^{2}(\bndry)\to T(\bndry).
\end{equation}
In this context the partial ND map is given as $
{\ND}^\Gamma_{\sigma,1}=\ND_{\sigma,1}Q.$ 
For the measurement we can only access the data at each electrode separately, to model this process appropriately, we use the orthogonal projection
\begin{equation}\label{eqn:ortho_proj}
P: g \mapsto \sum_{m=1}^M \frac{{\chi}_{_{m}}}{|E_m|} \int_{E_m} g \ ds, \quad L^{2}(\bndry)\to T(\bndry) 
\end{equation}
together with the linear extension operator
\begin{equation}\label{eqn:ortho_projExtension}
L: g \mapsto \sum_{m=1}^M {\widetilde{\chi}_{_{m}}} g_m, \quad T(\bndry)\to L^{2}(\bndry) . 
\end{equation}
It can be seen by straightforward calculations that $LP$ is the adjoint of $Q$ in $L^2(\bndry)$. 

Given an input current $\varphi\in\Hhalf$, the full measurement process is then modelled by 
\[
g=LP{\ND}^\Gamma_{\sigma}\varphi=LP\ND_{\sigma}Q\varphi.
\]
A similar construction has been already used in the computational study \cite{Alsaker2017}. This model can be seen as a continuum interpretation of an electrode setting and as we show in the following, it is more suitable for interpreting electrode measurements on a partial boundary. We will call the corresponding problem the {\em electrode continuum model} (ECM), summarized by
\begin{equation}
\label{eqn:ECM}
\left\{
\begin{array}{rcl}
\nabla\cdot\sigma\nabla u &=& 0 \mbox{ in }\Omega,\\
\sigma\dnu u &=& Q\varphi \mbox{ on } \Gamma, \\
Pu &=& U \mbox{ on }\Gamma.
\end{array}
\right.
\end{equation}
Following the formulation for the CEM in \eqref{eqn:CEM} we assume here as well that $\int_{\bndry}u\; ds=0$. It can be easily seen that $P{\ND}^\Gamma_{\sigma}$ is the corresponding measurement mapping that maps $\varphi$ to $U$; note first that $\ND_\sigma^\Gamma=\ND_\sigma Q$ and by definition we have that a solution $u$ of \eqref{eqn:introCondEq} with $\psi=Q\varphi$ solves the first two lines in \eqref{eqn:ECM}. Further, let $U=Pu$, then $u$ satisfies \eqref{eqn:ECM}.

In the following we analyse how \eqref{eqn:ECM} is connected to measurements from the CEM, but first we need to establish a few assumptions on the geometry. The most important property is that we do assume a lower bound for the ratio of the electrode size $|E|$ and the size of extended electrodes by
\begin{equation}\label{eq:elecSizeCond}
\frac{|\Gamma|}{|\bndry|}\geq \min_{m=1,\dots,M}\frac{|E|}{|\widetilde{E}_m|}\geq c_E>0.
\end{equation}
We note that in case of equidistant electrodes with equally sized extended electrodes, i.e. the distance between all neighbouring electrodes is the same, the first inequality is an equality. Anyhow, we are here particularly interested in the case were electrodes do not cover a certain part of the boundary and hence this will be a strict inequality and in this case we further need to assume that the electrode size is bounded from below by $|E|\geq c >0$. Given the condition \eqref{eq:elecSizeCond}, we can establish similar to \cite{Hyvoenen2009} that
\begin{equation}\label{eq:nonOrthobound}
\|Q\varphi\|_{L^2(\bndry)} \leq \frac{1}{\sqrt{c_E}}\|\varphi\|_{L^2(\bndry)} \text{ and }\|LP\varphi\|_{L^2(\bndry)}\leq \frac{1}{\sqrt{c_E}}\|\varphi\|_{L^2(\bndry)}.
\end{equation}
We note that $P$ is an orthogonal projection and hence the constant in the second bound is due to the linear extension $L$.

The following results summarize the relation of the projections $Q$ and $P$ in case of functions supported on the partial boundary $\Gamma$.
\begin{lemma}\label{lem:selfAd}
Let $\varphi\in L^2(\bndry)$ with $\supp(\varphi)\subset\Gamma$, then $Q\varphi=P\varphi$. Furthermore, it holds that
\begin{equation}\label{eq:selfAdjointness}
(P\varphi,\psi)_{L^2(\bndry)}=(\varphi,P\psi)_{L^2(\bndry)}, \hspace{0.25 cm} \text{ for all } \psi \in L^2(\bndry).
\end{equation}
\end{lemma}
\begin{proof}
The identity $Q\varphi=P\varphi$ follows directly from the definition. The identity \eqref{eq:selfAdjointness} can be shown by direct calculations. Let $\varphi\in L^2(\bndry)$ with $\supp(\varphi)\subset\Gamma$ and $\psi\in L^2(\bndry)$, then
\begin{align*}
(P\varphi,\psi)_{L^2(\bndry)}&=\int_{\bndry}\sum_{m=1}^M \frac{\chi_m(x)}{|E_m|}\int_{E_m} \varphi(y) ds_y \psi(x) ds_x \\
&=\sum_{m=1}^M \frac{1}{|E_m|}\int_{E_m}\chi_m(x)\psi(x)ds_x \int_{E_m} \chi_m(y)\varphi(y) ds_y  \\
&=\sum_{m=1}^M \int_{E_m}\frac{\chi_m(y)}{|E_m|}\int_{E_m}\psi(x)ds_x \varphi(y) ds_y  \\
&= \int_{\bndry}\sum_{m=1}^M\frac{\chi_m(y)}{|E_m|}\int_{E_m}\psi(x)ds_x \varphi(y) ds_y  =(\varphi,P\psi)_{L^2(\bndry)}.
\end{align*}
\end{proof}


Now we can establish the error estimate between measurements of the ECM and CEM. The following theorem shows that the approximation error depends linearly on the size of the electrodes. Since we are interested in a partial-boundary setting, the essential result is that the error does not depend on the missing boundary.
\begin{theorem}\label{theo:approx}
Given the measurement operator $\ND^{E}_\sigma$ for the complete electrode model \eqref{eqn:CEM} and $P\ND^\Gamma_\sigma=P\ND_\sigma Q$ for the electrode continuum model \eqref{eqn:ECM}. Let the electrodes be such that \eqref{eq:elecSizeCond} is satisfied, then the following estimate holds
\begin{equation}\label{eqn:operatorIden}
\|LP\ND^\Gamma_\sigma-L(\ND^{E}_\sigma-zI)Q\|_{L^2(\bndry)\to L^2(\bndry)} \leq \frac{C|E|}{{c_E}}.
\end{equation}
\end{theorem}
\begin{proof}
This proof is an adjustment of \cite[Theorem 4.1]{Hyvoenen2009} to our setting, we summarize the essential parts and present changes. In particular, we show that for any function $\varphi\in L^2_\diamond(\bndry)$ that the $L^2$-norm is bounded by 
\begin{equation}\label{eqn:normDiffOperat}
\|(LP\ND^\Gamma_\sigma-L(\ND^{E}_\sigma-zI)Q)\varphi\|_{L^2(\bndry)}\leq \frac{C|E|}{{c_E}}\|\varphi\|_{L^2(\bndry)}.
\end{equation}
The solutions of the ECM \eqref{eqn:ECM} are denoted by $u_0$ and for the CEM \eqref{eqn:CEM} by $u$. By \cite{Hyvoenen2009} we have that $L\ND^E_\sigma Q\varphi=LP(u|_{\bndry})+zLQ\varphi$. This leaves us to estimate
\[
\|(LP\ND^\Gamma_\sigma-L(\mathcal{R}^{E}_\sigma-zI)Q\varphi\|_{L^2(\bndry)} = \|LP(u_0 - u)\|_{L^2(\bndry)}.
\]

\noindent
We now use the boundedness of $LP$ \eqref{eq:nonOrthobound} and estimate the solutions on the boundary by their corresponding Neumann data (boundedness of the single layer operator, see for instance \cite[Lemma 7.1]{Nachman1996}), then we obtain
\begin{align*}
\nonumber \|LP(u_0 - u)\|_{L^2(\bndry)} &\leq \frac{1}{\sqrt{c_E}} \|u_0 - u\|_{L^2(\bndry)} 
\leq \frac{C}{\sqrt{c_E}} \|\sigma\dnu (u_0 - u)\|_{H^{-1}(\bndry)}.
\end{align*}
By the boundary condition of the ECM we have $\sigma\dnu u_0 = Q\varphi$. Further, we note that the Neumann data are both supported on $\Gamma$ and using Lemma \ref{lem:selfAd}, we obtain
\begin{align*}
\nonumber \|\sigma\dnu (u_0 - u)\|_{H^{-1}(\bndry)}  &= \| Q\varphi -\sigma\dnu u\|_{H^{-1}(\bndry)}
 = \|(P-I) \sigma\dnu u\|_{H^{-1}(\bndry)} .
\end{align*}
Using the definition of the $H^{-1}$-norm by its dual space and by \eqref{eq:selfAdjointness} in Lemma \ref{lem:selfAd} we get
\begin{align*}
\|(P-I)\sigma\dnu u\|_{H^{-1}(\bndry)}&=\sup_{\|\psi\|_{H^1(\bndry)}=1} \int_{\bndry} (P-I)\sigma\dnu u  \psi \ ds \\
&=\sup_{\|\psi\|_{H^1(\bndry)}=1} \int_{\bndry} \sigma\dnu u (P-I)\psi \ ds \\ 
&\leq \sup_{\|\psi\|_{H^1(\bndry)}=1} \|\sigma\dnu u \|_{L^2( \Gamma)} \| (P-I) \psi\|_{L^2( \Gamma)}.
\end{align*}
The last term can be estimated using a one-dimensional Poincar{\'e} inequality \cite{Lieb2003} on each electrode separately, together with the bound on $\psi\in H^1(\bndry)$ this gives
\[ 
\sup_{\|\psi\|_{H^1(\bndry)}=1} \| (P-I) \psi\|_{L^2( \Gamma)}\leq C |E|.
\]
By \cite[Lemma 2.1]{Hyvoenen2009} it holds that $\|\sigma\dnu u \|_{L^2( \bndry)}\leq C\|Q\varphi\|_{L^2( \bndry)}$ and we obtain with \eqref{eq:nonOrthobound} the claim
\[
\|LP(u_0 - u)\|_{L^2(\bndry)} \leq 
\frac{C|E|}{\sqrt{c_E}} \|Q\varphi \|_{L^2( \bndry)}\leq \frac{C|E|}{{c_E}} \|\varphi \|_{L^2( \bndry)}.
\]
\end{proof}
The connection to \cite[Theorem 4.1]{Hyvoenen2009} can be seen easily if we assume the electrodes to be equidistant with $|E|\to 0$ and $|\Gamma|\to|\bndry|$, that means the amount of electrodes $M\to\infty$. The important difference in the above estimate is that our estimate does not depend on the missing boundary and hence under constant electrode size it stays stable if larger parts of the domain are not covered.

On a more practical issue, given a fixed number of electrodes, in order to maximize the signal-to-noise ratio one should use larger electrodes \cite{Gisser1990} as well as if the electrode size becomes small, the power required to push a current through becomes large \cite{Cheney1992}. That means in practice the electrodes won't be too small. Thus, the estimate \eqref{eqn:operatorIden} states that independent of the measurement domain and amount of electrodes used (for fixed $c_E$), that the proposed model is a stable approximation in the continuum setting for electrode measurements, that can be used for further analysis.

Furthermore, the result can be extended to difference data, if one assumes that the contact impedance does not change for different conductivities and between measurements. Then we are left with
\[
\|LP\ND^\Gamma_{\sigma,1}-L\ND^{E}_{\sigma,1}Q\|_{L^2(\bndry)\to L^2(\bndry)} \leq \frac{C |E|}{c_E}.
\]

\section{Formulating the optimization problem}\label{sec:OptProb}

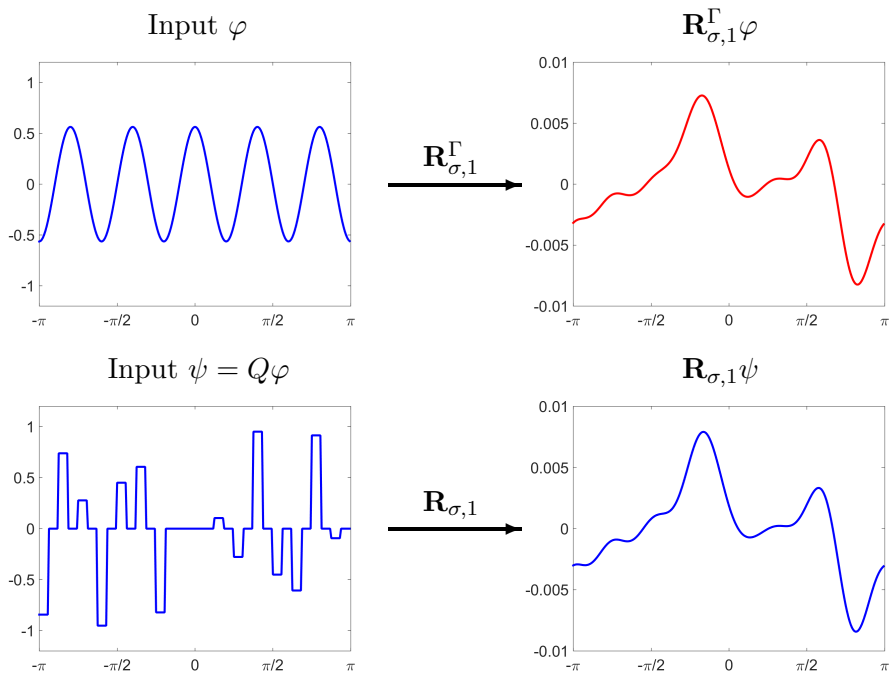
\begin{figure}[t!]
\centering
\begin{picture}(350,250)
\put(0,130){\includegraphics[width=150 pt]{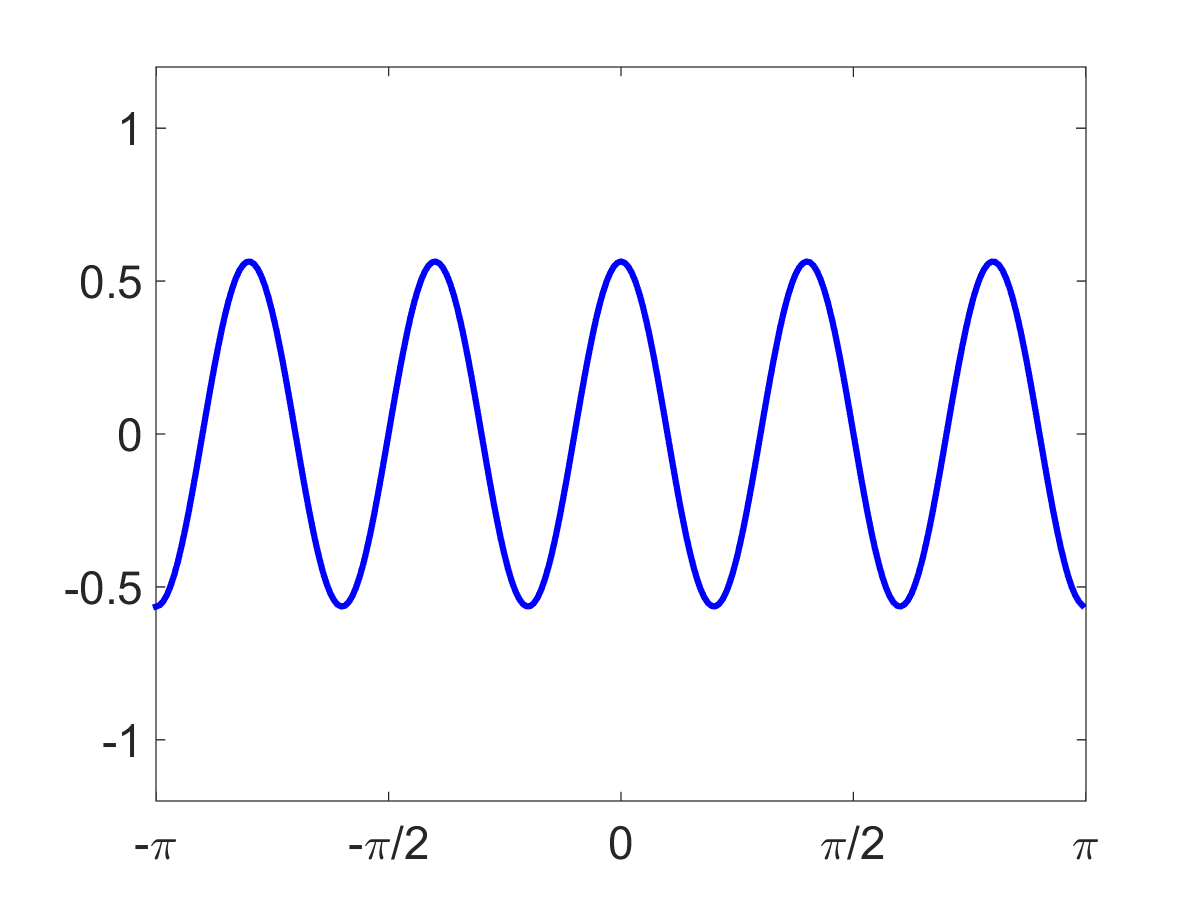}}
\put(200,130){\includegraphics[width=150 pt]{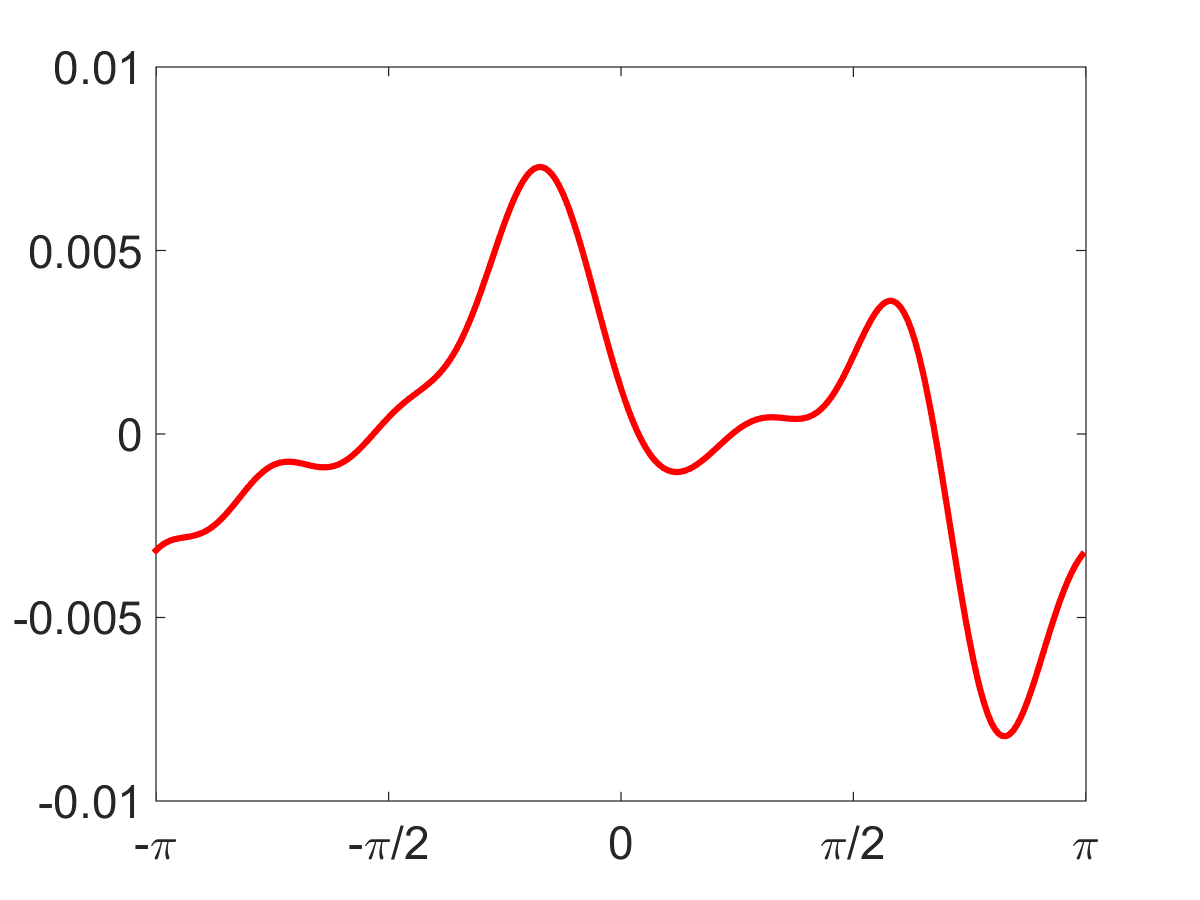}}
\put(0,0){\includegraphics[width=150 pt]{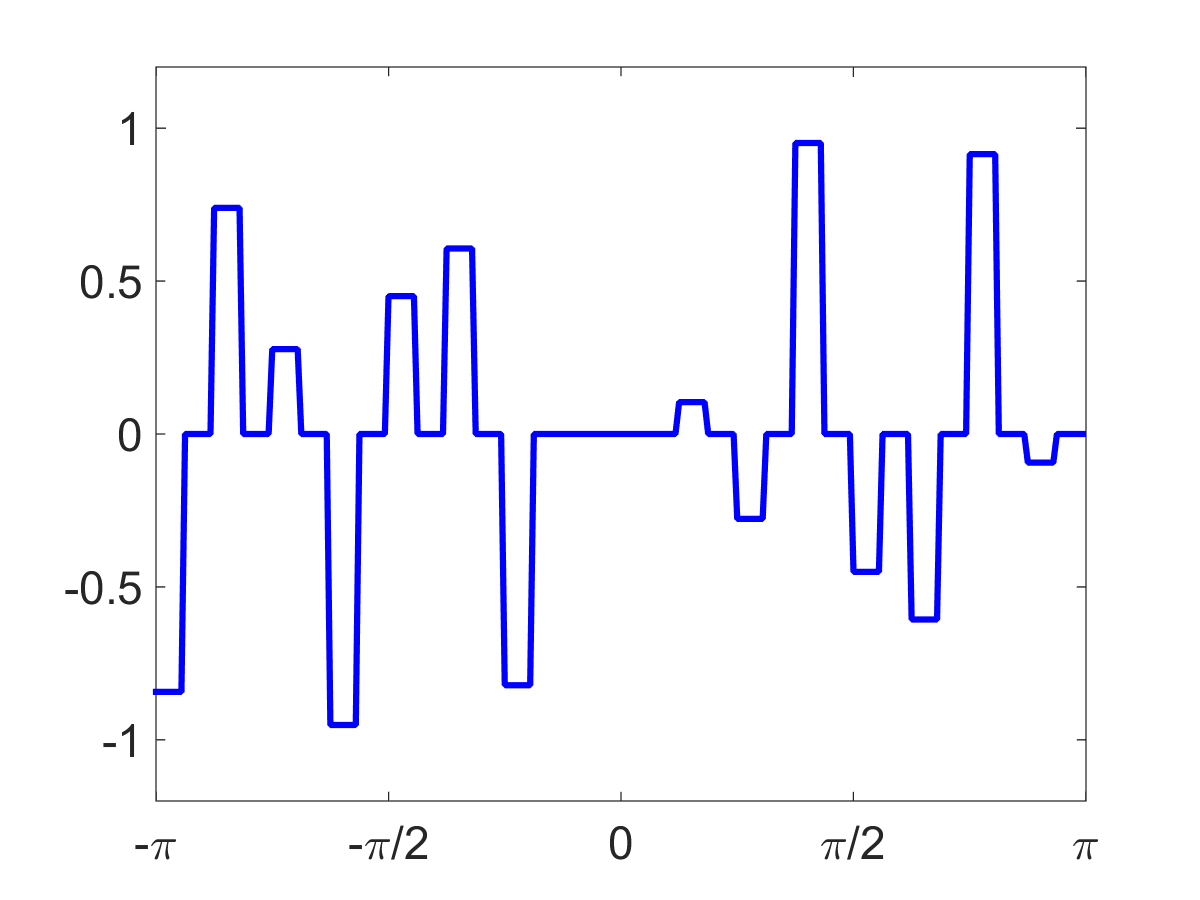}}
\put(200,0){\includegraphics[width=150 pt]{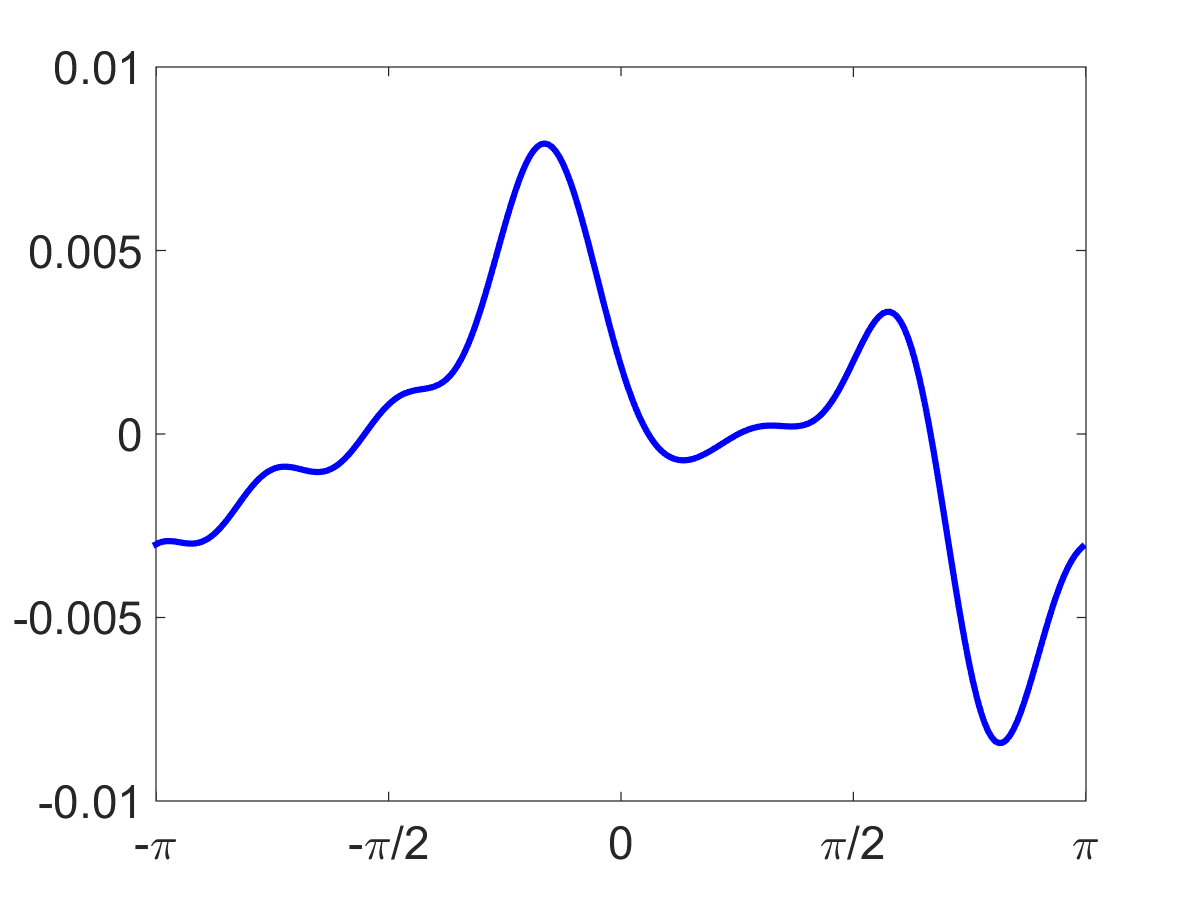}}
\put(60,245){Input $\varphi$}
\put(45,115){Input $\psi=Q\varphi$}

\put(163,195){$\mathbf{R}^\Gamma_{\sigma,1}$}
\put(163,65){${\mathbf{R}}_{\sigma,1}$}
\thicklines
\linethickness{.4mm}
\put(150,188){\vector(1,0){50}}
\thicklines
\linethickness{.4mm}
\put(150,58){\vector(1,0){50}}

\put(260,115){${\mathbf{R}}_{\sigma,1}\psi$}
\put(260,245){${\mathbf{R}}^\Gamma_{\sigma,1}\varphi$}

\end{picture}
\caption{\label{fig:mapping}Illustration of the central identity for the basis function $\varphi_5(\theta)=\cos(5\theta)/\sqrt{\pi}$ (top left) and its nonorthogonal projection to 14 electrodes with a gap around the angular value $\theta=0$ (bottom left). The two traces on the right represent the results after applying the ND map (bottom) and the partial ND map (top).} 
\end{figure} 
In this section we will discuss how to approximate the ND map corresponding to full-boundary data from partial-boundary measurements, or rather the matrix approximation of the ND map from partial-boundary measurements. The following derivations are based on measurements from the ECM, but can be applied to real electrode data as motivated by Theorem \ref{theo:approx}.
Let us for now omit the extension operator $LP$ and only consider the partial ND map ${\ND}^\Gamma_\sigma$ and the ND map $\ND_\sigma$. We remind that in the central identity we have
$
\ND_\sigma\psi={\ND}^\Gamma_\sigma \varphi,
$
where we know $\varphi$, $\psi$ and we can measure ${\ND}^\Gamma_\sigma$. That suggests we can recover some information about the ND map.  As illustrated in Figure \ref{fig:mapping}, the two traces on the right are rather similar and hence we can deduce some characteristics of the ND matrix $\NDdiffMat$. We consider in the following the measured difference maps given as matrices $\NDdiffMat$ and $\pNDdiffMat$ and the nonorthogonal projection $Q$ defined in \eqref{eqn:nonortho_proj}. We would like to note, that the construction works also for different choices of partial-boundary maps. Let the orthonormal basis functions of order $n\in\mathcal{N}=\{-N/2,\dots -1,1,\dots,N/2\}$, for $N\in\N$ even, be given by
\[
\varphi_n(\theta)=\left\{
\begin{array}{cl}
\frac{1}{\sqrt{\pi}}\sin(n\theta) &\text{ if } n<0,\\
\frac{1}{\sqrt{\pi}}\cos(n\theta) &\text{ if } n>0.
\end{array}\right.
\]
The measurement for the basis $\varphi_n$ is then given by $g_n:=\NDdiff Q\varphi_n$ and we denote the partial-boundary current by $\psi_n=Q\varphi_n$. The first approach that comes to ones mind is to find $\NDdiff$ as minimizer of
\begin{equation}\label{eqn:firstOptApp}
\min_{\NDdiff} \sum_{n\in \mathcal{N}} \|\NDdiff\psi_n-\pNDdiff\varphi_n\|^2_{L^2(\bndry)}.
\end{equation}
We want to formulate \eqref{eqn:firstOptApp} as a more reasonable optimization problem in a finite dimensional space with the matrix $\NDdiffMat$ sought for. As discussed in Section \ref{sec:ModelCont}, the matrices are given with respect to the chosen basis as
\begin{align*}
(\pNDdiffMat)_{n,j}&=(\pNDdiff \varphi_n,\varphi_j)_{L^2(\bndry)}=(\NDdiff \psi_n,\varphi_j)_{L^2(\bndry)},  \\
(\NDdiffMat)_{n,j}&=(\NDdiff \varphi_n,\varphi_j)_{L^2(\bndry)}.
\end{align*}
Further, we need the coefficients of the basis functions and the measurement:
\begin{align*}
(\hat{\psi}_n)_j &= ( \psi_n,\varphi_j)_{L^2(\bndry)}=(Q \varphi_n,\varphi_j)_{L^2(\bndry)},  \\
(\hat{g}_n)_j &= ( g_n,\varphi_j)_{L^2(\bndry)}=(\NDdiff Q\varphi_n,\varphi_j)_{L^2(\bndry)} .
\end{align*}
We note that in fact $(\pNDdiffMat)_{n,j}=(\hat{g}_n)_j$. Now, given the coefficients of the nth basis function $\hat{\psi}_n$, an approximate application of the ND map is given as $\NDdiffMat\hat{\psi}_n$, or written by its transpose as $(\hat{\psi}_n^T\NDdiffMat^T)^T$. The transposed representation can be also written with a matrix $\Psi_n\in\R^{N\times N^2}$ for $\hat{\psi}_n$ and a vector $r\in\R^{N^2}$ for $\NDdiffMat$, defined as follows
\begin{equation}\label{eqn:vectorization}
\Psi_{n}=\left(
\begin{array}{ccc}
\hat{\psi}^T_{n} &  & 0 \\
 & \ddots &  \\
0 &  & \hat{\psi}^T_{n} 
\end{array} \right), \hspace{0.25cm} \text{ and }
r=\left(
\begin{array}{c}
r_1^T \\
\vdots \\
r_N^T 
\end{array} \right),
\end{equation}
where $r_1,\dots,r_N$ are the rows of $\NDdiffMat$. In the following we want to write the application of the ND map for all basis functions at once. 
For this purpose, let $\Psi\in\R^{N^2\times N^2}$ be the matrix consisting of all $\Psi_n$ for all $n\in \mathcal{N}$ stacked, and similarly $g\in\R^{N^2}$ consisting of all $\hat{g}_n$. Then the finite dimensional equivalent of \eqref{eqn:firstOptApp} can be written in matrix-vector notation with respect to $N$ basis functions as the minimization problem
\begin{equation}\label{eq:FidelityOnly}
\min_{r} \|\Psi r - g\|_2^2.
\end{equation}
As well known, a solution of \eqref{eq:FidelityOnly} can be found by solving the normal equation
\begin{equation}
\label{eq:NormalEquation}
\Psi^T\Psi r=\Psi^T g.
\end{equation}
We will prove next that in a special case of rotationally symmetric conductivities with certain restrictions, that this simple construction is already sufficient to uniquely determine the ND matrix for full boundary data as solution of \eqref{eq:NormalEquation}. The more general case will be discussed in Section \ref{sec:Approx2ND}. Let us first note that with the nonorthogonal projection as partial-boundary map one has under certain geometric assumptions that $\Psi$ is square with full rank. In our setting this is the case if the number of basis functions is strictly less than the number of electrodes, which are equally spaced and sized, and $\Omega$ is a disk.

\begin{lemma}\label{lem:reconRotSym}
Let $\sigma\in L^\infty(\Omega)$ be rotationally symmetric, that means $\sigma(z)=\sigma(|z|)$. Further, let $\Omega\subset\R^2$ be the unit disk and $\sigma \equiv 1$ close to the boundary, and assume $\Psi$ has full rank. Given the measurement $g_n=\NDdiff \psi_n$, then \eqref{eq:NormalEquation} has a unique solution that coincides with the full-boundary ND matrix of order $N$ with $N<M$, where $M$ is the number of electrodes.
\end{lemma}
\begin{proof}
We show that the unique solution of
\begin{equation}\label{eq:rotUnique}
\Psi r = g
\end{equation}
coincides with the ND matrix $\NDdiffMat$ if the conductivity is rotationally symmetric.
Let $n\in\mathcal{N}$ and we consider the measurement $g_n$, for which the coefficients are given by $(\hat{g}_n)_j=(\NDdiff \psi_n,\varphi_j)$.

The basis functions $\varphi_n$ are eigenfunctions of the ND maps $\ND_\sigma$ and $\ND_1$ \cite{Gisser1990,Mueller2012}, hence $\NDmat_\sigma$ and $\NDmat_1$ are both diagonal matrices. Then by linearity and self-adjointness we get
\begin{equation}\label{eq:coefRotSym}
(\hat{g}_n)_j=(\NDdiff\psi_n,\varphi_j)=(\psi_n,\NDdiff\varphi_j)=c_j(\psi_n,\varphi_j) = c_j (\hat{\psi}_n)_j,
\end{equation}
where $c_j$ is the eigenvalue of $\NDdiff$ corresponding to $\varphi_j$.
Thus, the coefficient of the measurement $(\hat{g}_n)_j$ is just the basis function $(\hat{\psi}_n)_j$ multiplied by the eigenvalue $c_j$ of $\NDdiff$. We note that due to the full rank assumption at least one $(\hat{\psi}_n)_j\neq 0$ for each $n\in\mathcal{N}$.

Let us now write $\NDdiffMat=\text{diag}(c_{-N/2},\dots,c_{N/2})$ as diagonal matrix and let $r=\NDdiffMat(:)$ be the vectorized form, inserting this with \eqref{eq:coefRotSym} into \eqref{eq:rotUnique} we see that $r$ is a solution.  By construction we have that $\Psi$ is square and due to the full rank assumption there exists a unique inverse, hence $r$ is uniquely defined. 
\end{proof}

This result shows that one could also just solve $\Psi r = g $. But with respect to the electrode extension operator $LP$ as well as numerical noise, this is not recommended. Further from the proof we can see that in fact for rotationally symmetric conductivities we need a lot less data.
\begin{corollary}
Under the assumptions of Lemma \ref{lem:reconRotSym} and additionally let $(\hat{\psi}_n)_j\neq 0$ for all $j\in\mathcal{N}$, then for one fixed $n\in\mathcal{N}$ we can recover the ND matrix $\NDdiffMat$ of order $N$ from one measurement $g_n=\NDdiff \psi_n$.
\end{corollary}
\begin{proof}
Given the measurement coefficients $(\hat{g}_n)_j= c_j (\hat{\psi}_n)_j,$ and
\[
\Psi=\left(
\begin{array}{ccc}
(\hat{\psi}_{n})_{-N/2} & \cdots & 0 \\
 & \ddots &  \\
0 & \cdots & (\hat{\psi}_n)_{N/2} 
\end{array} \right),
\]
then $r=(c_{-N/2},\dots,c_{N/2})$ solves \eqref{eq:rotUnique} and  $\NDdiffMat=\text{diag}(r).$
\end{proof}

\begin{figure}[t!]
\centering
\begin{picture}(320,150)
\put(-30,-10){\includegraphics[width=200pt]{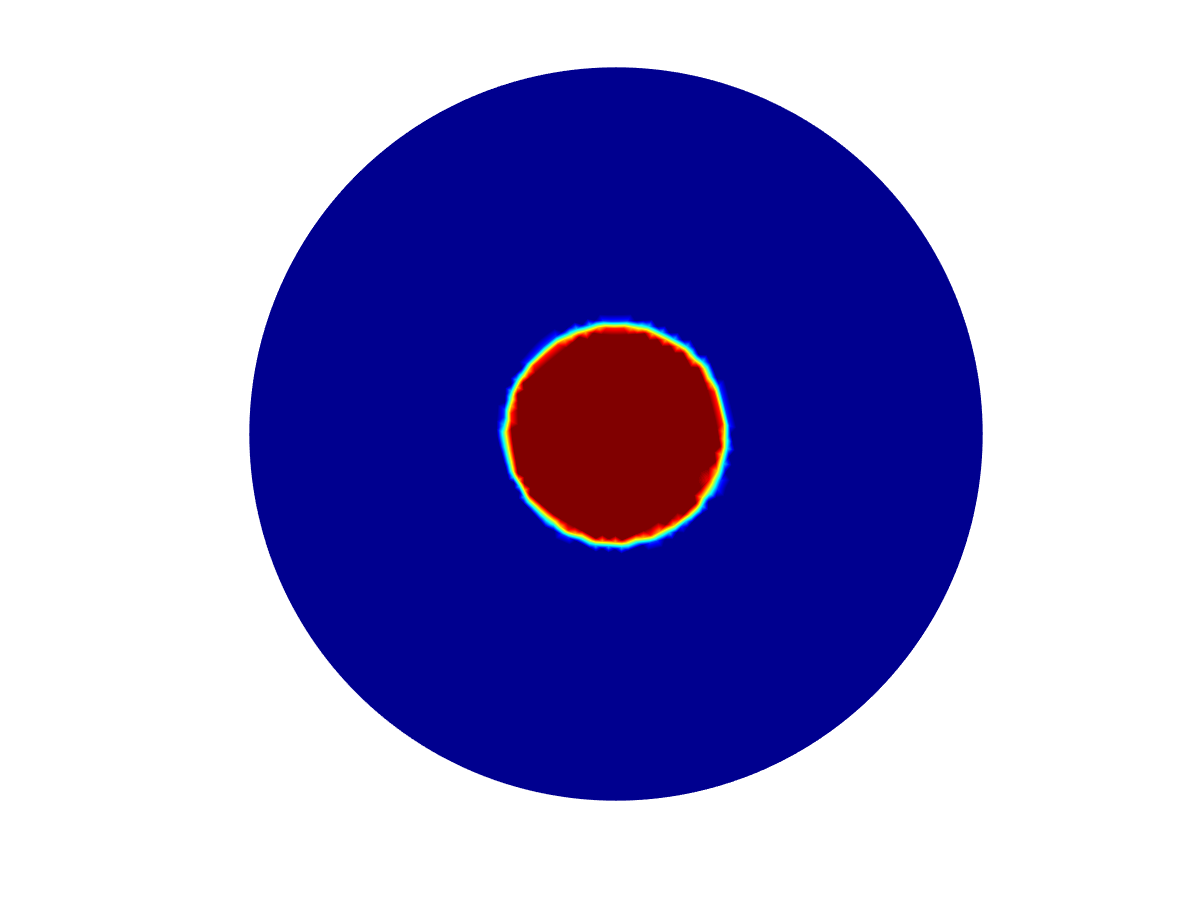}}
\put(140,-10){\includegraphics[width=200pt]{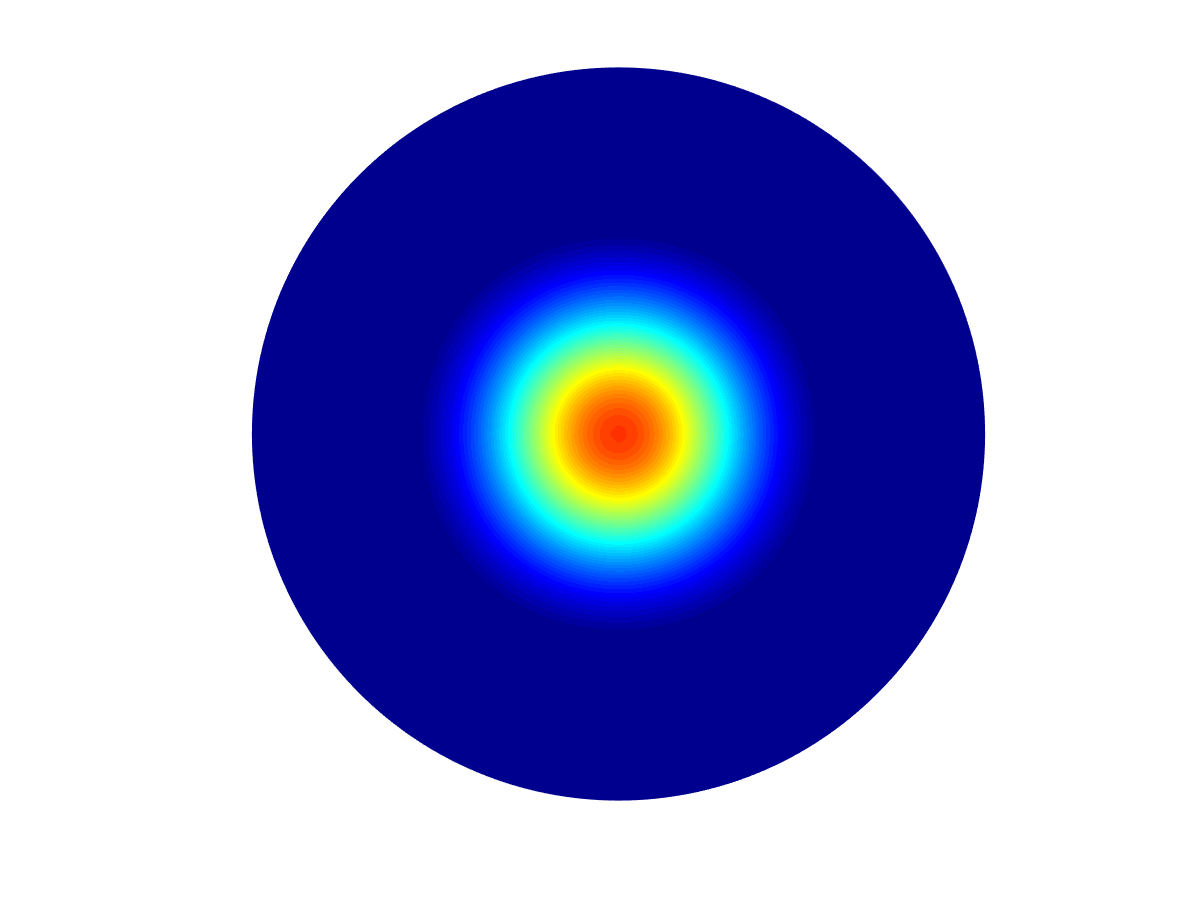}}
\put(50,140){\textcolor{black}{Phantom}}
\put(210,140){\textcolor{black}{Reconstruction}}

\end{picture}
\caption{\label{fig:rotSym} Illustration of Corollary 3.2: Reconstruction of a rotationally symmetric conductivity with the D-bar algorithm from a single measurement $g_2=\NDdiff \psi_2$, where $\psi_2=Q\varphi_2$ and $\varphi_2(\theta)=\cos(2\theta)$.}
\end{figure} 

The purpose of Lemma \ref{lem:reconRotSym} is to illustrate that the reduction to a minimization problem is a sensible approach. We note that the additional assumption is easily satisfied in the partial-boundary setting. Furthermore, it is important to note that for more general conductivities it is not sufficient any more to just solve the normal equation \eqref{eq:NormalEquation}. Furthermore, if we have only electrode measurements available we can not expect to find a unique solution anymore, due to the discrepancy of the ECM as described in Theorem \ref{theo:approx} and especially with additional measurement noise in mind. Thus, we will discuss in the following section how to extend this approach for general conductivities and data acquired from electrodes.

\subsection{Finding an approximation to the ND map}\label{sec:Approx2ND}
In the following we propose an algorithm that is capable to compute an approximation to the ND matrix $\NDdiffMat$  from electrode input. We start by discussing how to compute an approximation to continuum data from the electrode measurements and then use this acquired data in the framework of Section \ref{sec:OptProb} to obtain the approximated ND map.

\subsubsection{Approximating continuum data from electrode measurements} 
The first task is to recover the continuum data $g_n=\NDdiff\psi_n$ from possibly noisy electrode measurements modelled by $U^n=LPg_n$ for all $n\in\mathcal{N}$. Here real data or measurements from the CEM are considered as noisy input.
We note that the adjoint of $LP$ in $L^2(\bndry)$ is given by the nonorthogonal electrode projection $Q$. That means we have
\[
(LP\varphi,\psi)_{L^2(\bndry)}=(\varphi,Q\psi)_{L^2(\bndry)}, \text{ for } \varphi,\psi \in L^2(\bndry).
\]
This fact will be especially helpful to recover a suitable $g'$ such that $LPg' \approx U^n$. It is easy to see, that the solution of $LPg'=U^n$ is not unique, hence we need some further information. For this we make use of the fact that difference data is smooth, as discussed for instance in \cite{Hanke2011}. Thus, it is reasonable to search for $g'$ as the minimizer of a Tikhonov functional 
\begin{equation}\label{eq:tikhonov4trace}
\|LPg'-U^n\|^2_{L^2(\bndry)} + \alpha \| g' \|^2_{L^2(\bndry)}.
\end{equation}
The minimizer is given as the solution of the regularized normal equation
\begin{equation}\label{eq:normalEq4trace}
(QLP+\alpha I)g'=QU^n.
\end{equation}
Numerically the solution can be computed matrix free, for instance by the conjugate gradient method, if one has $Q$ and $LP$ as functions implemented. The parameter $\alpha>0$ controls the smoothness of the solution and can be adjusted by a priori knowledge of the measured object. An example result of this procedure is illustrated in Figure \ref{fig:approxContinuumTrace}.

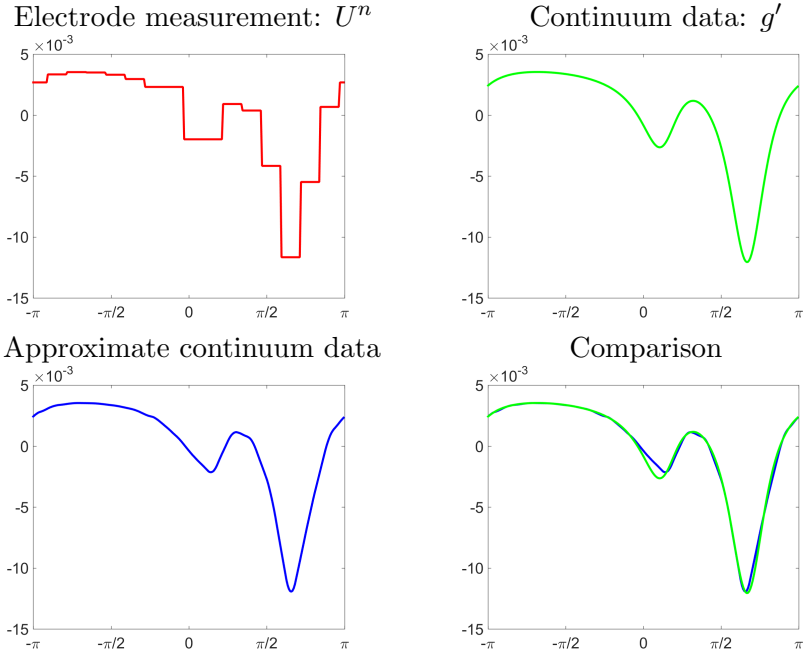
\begin{figure}[t!]
\centering
\begin{picture}(320,250)
\put(-10,135){\includegraphics[width=150pt]{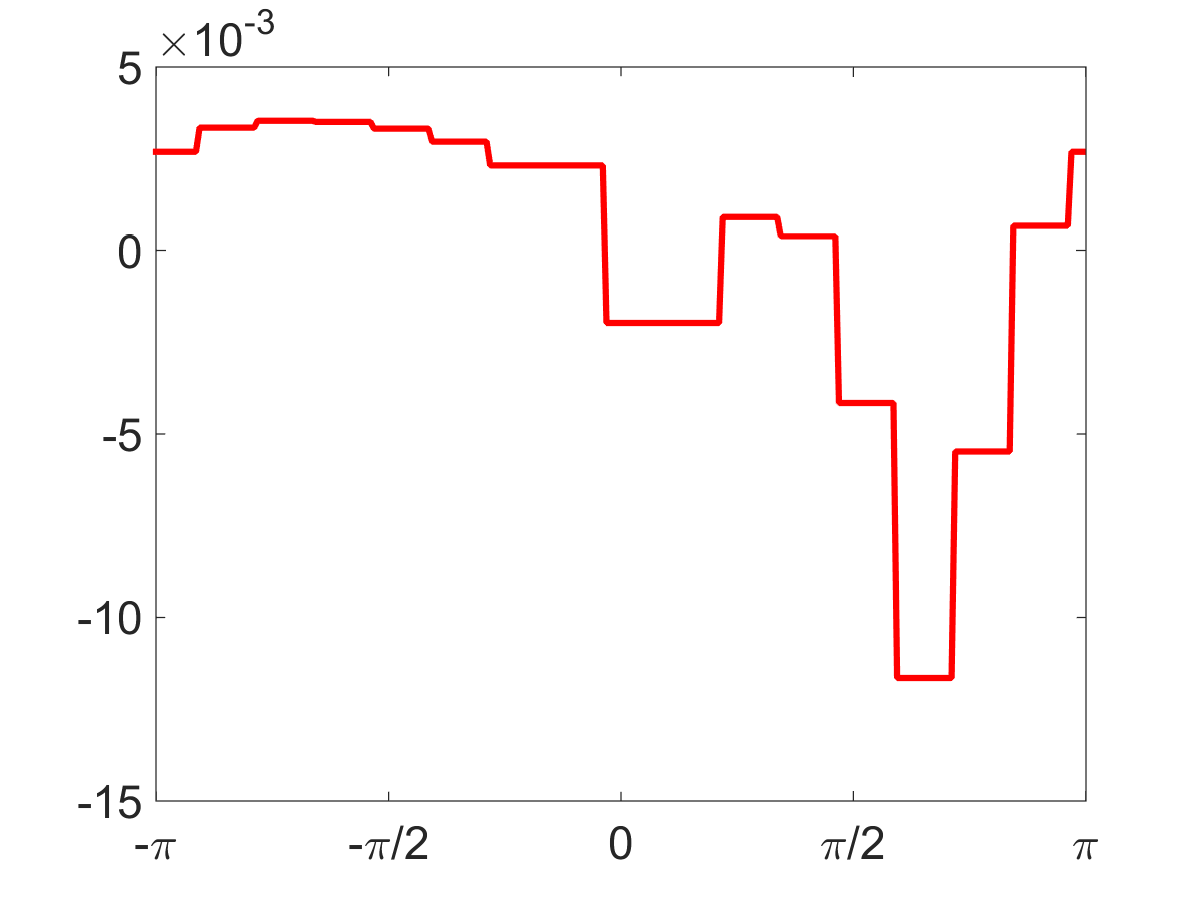}}
\put(160,135){\includegraphics[width=150pt]{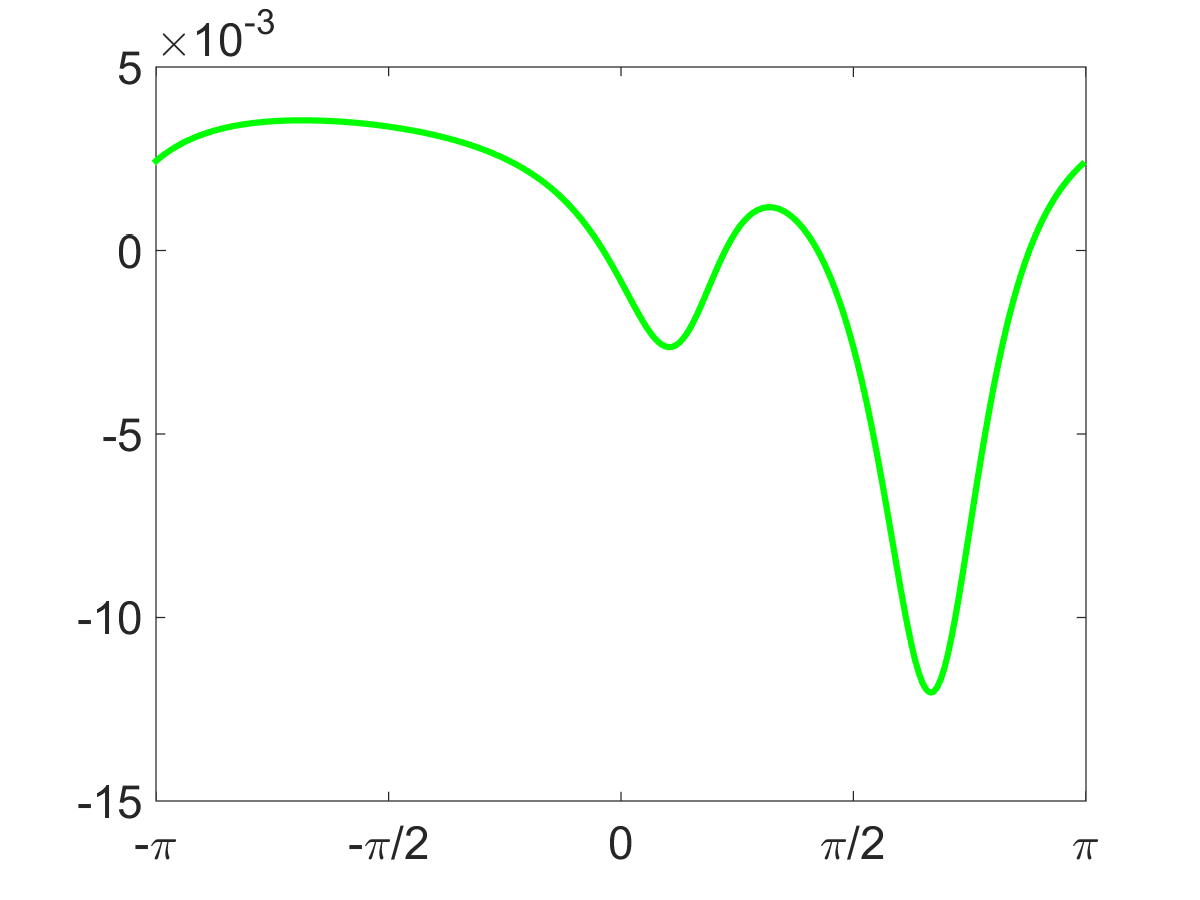}}
\put(-10,10){\includegraphics[width=150pt]{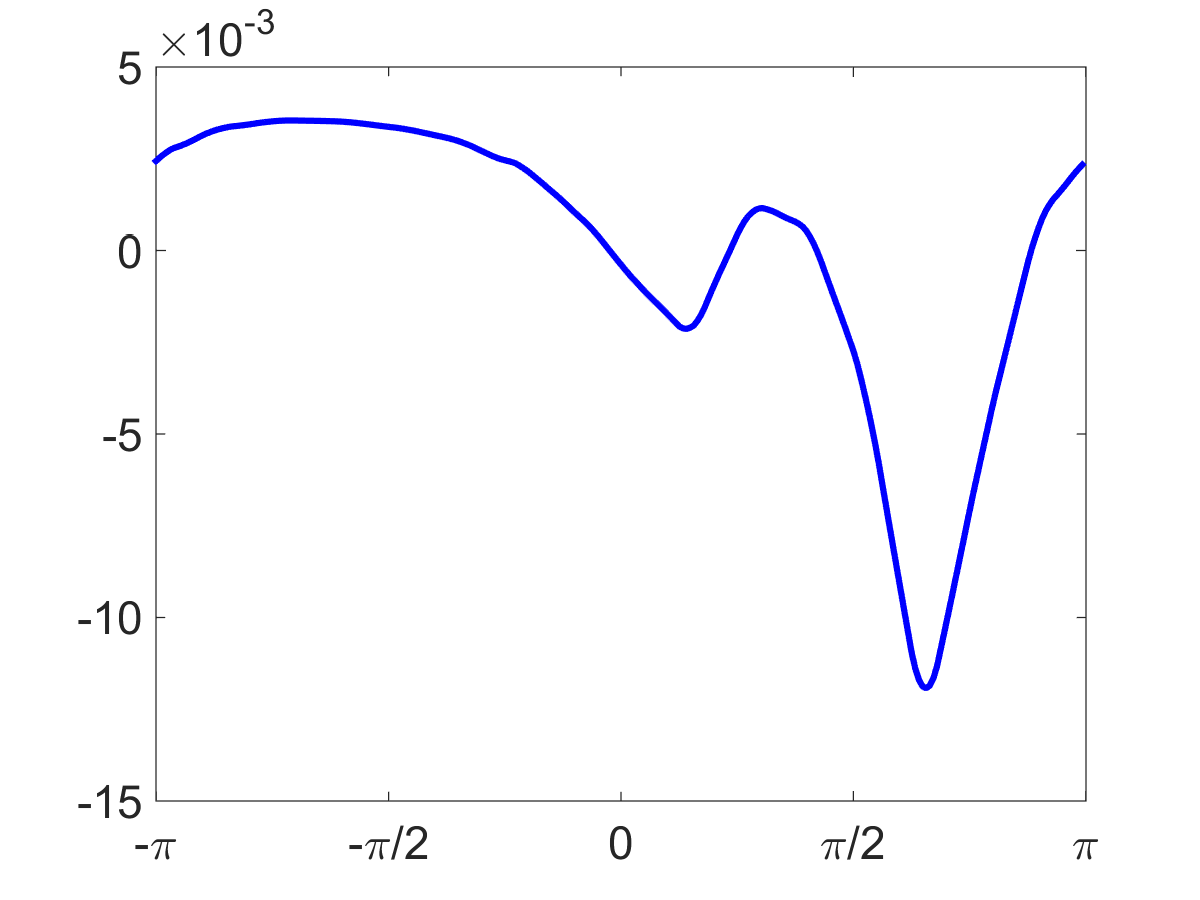}}
\put(160,10){\includegraphics[width=150pt]{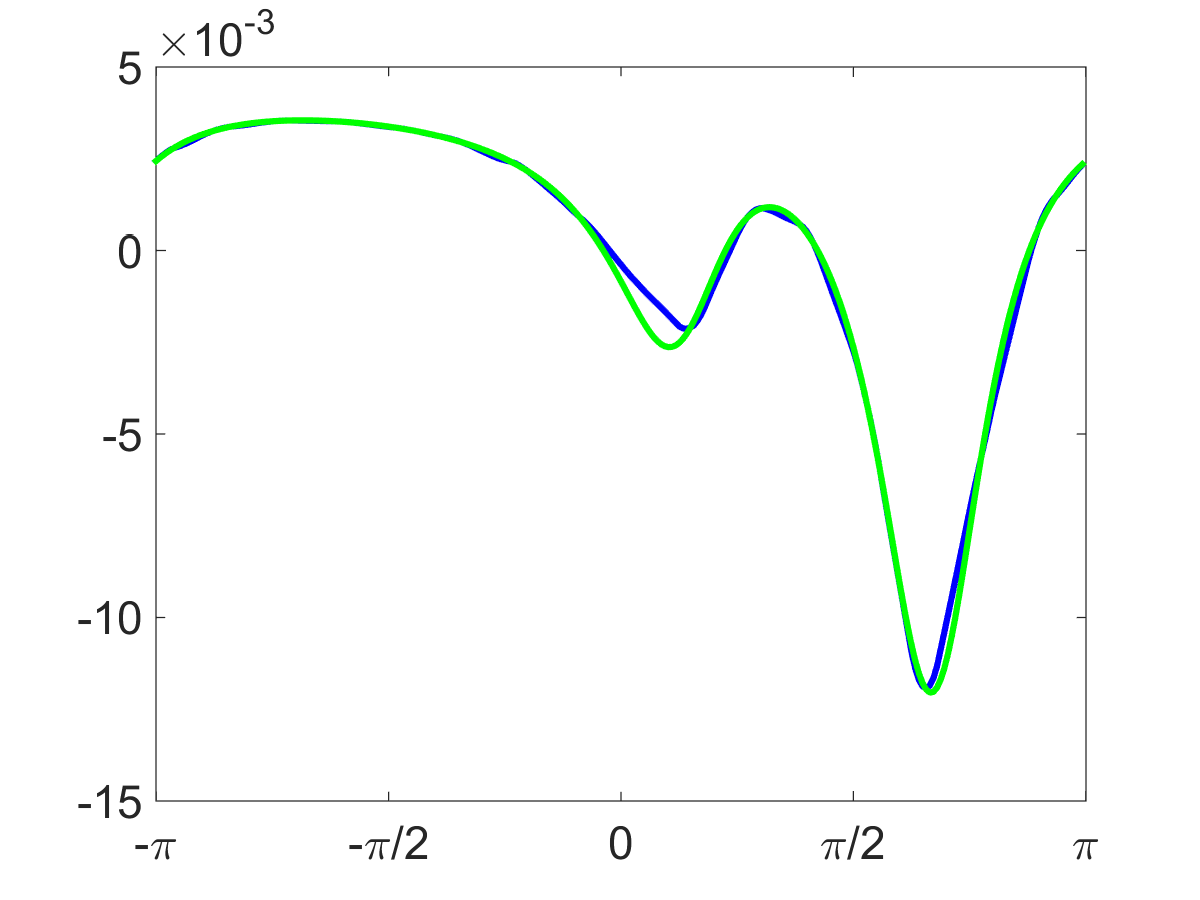}}

\put(2,250){\textcolor{black}{Electrode measurement: $U^n$}}
\put(195,250){\textcolor{black}{Continuum data: $g'$}}
\put(210,125){\textcolor{black}{Comparison}}
\put(-2,125){\textcolor{black}{Approximate continuum data}}
\end{picture}
\caption{\label{fig:approxContinuumTrace}  Illustration of the approximation procedure to obtain continuum data. The top shows electrode data (top left) and the corresponding idealized continuum data (top right). The approximation (bottom left) is computed by solving \eqref{eq:normalEq4trace} and a comparison to continuum data is shown (bottom right).}
\end{figure} 

\subsubsection{Computing an approximate ND map}
In practice, for $M$ active electrodes we can only achieve $M-1$ degrees of freedom for a linearly independent basis. Depending on the partial-boundary map even less are possible. It is not in all cases advisable to restrict oneself to just $M-1$ degrees of freedom. If more basis functions are used one has an underdetermined system that is not uniquely solvable anymore, hence we need to impose constraints to the minimization problem \eqref{eq:FidelityOnly} to favor solutions of a certain kind. Also with respect to approximated continuum data and possibly noisy measurement, we need to impose some regularity on the solution.
We propose to use as constraint a simple 2-norm distance to the measured partial ND map, in vectorized form $r^\Gamma$ of $\pNDdiff$ as in \eqref{eqn:vectorization}, and combined with a weight matrix $S$. That is, we search for a solution of
\begin{equation}\label{eqn:oneStepOpt}
\min_{r} \|\Psi r - g\|_2^2 + \beta\|r-Sr^\Gamma\|_2^2.
\end{equation}
The solutions can be computed as the unique solution of
\[
r=(\Psi^T\Psi + \beta\ I)^{-1}(\Psi^T g+\beta S r^\Gamma).
\]
This leads to the difficulty of choosing the regularization parameter $\beta>0$ adequately as well as the weight matrix $S$. The matrix $S$ can be chosen with knowledge about the measured object and by that the expected structure of $\NDdiffMat$. 
As it will be shown in the computational section, this simple approach is already sufficient to demonstrate the effectiveness of the methodology.

\subsubsection{The resulting algorithm for electrode data}\label{sec:Algo}
Let us summarize at this point the individual steps to obtain a matrix approximation to the ND map from partial-boundary electrode measurements. Given any current input produced by a partial-boundary map, e.g. the nonorthogonal projection $Q$, then the resulting voltage measurements are modelled by $U^n=LP\NDdiff^\Gamma\varphi_n$. Measurements from the CEM or real data is treated as noisy input.
We aim to recover an approximation to the (full-boundary) ND map $\NDdiffMat$ by the following procedure:
\\
\begin{itemize}
\item[ 1.] Measure voltages $U^n$ for all input currents $Q\varphi_n$.
\item[ 2.] Approximate continuum traces $g_n$ from the measured voltages by \eqref{eq:tikhonov4trace}.
\item[ 3.] Find approximation to full-boundary ND matrix $\NDdiffMat$ by solving \eqref{eqn:oneStepOpt} \\
\end{itemize}
Now that the algorithm to obtain approximate ND maps has been stated, we will test the effectiveness of this approach on simulated noisy data and real data in the next section.

\section{Computational results}\label{sec:CompResults}
The essential question in electrical impedance tomography is at the end, how valuable are the reconstructions from the acquired data. Following this, we evaluate the approximate ND matrix by the reconstruction we can obtain from it, additionally we present an error table in Section \ref{sec:discus}. We will summarize very briefly the D-bar algorithm used for the reconstruction, it is specifically formulated for ND maps and was recently introduced in \cite{Hauptmann2017}. The strength of D-bar algorithms lies in their direct nature, that means there is just one unique solution that converges to the real conductivity if the noise goes to zero. In other words, the D-bar method is a regularization strategy (for full-boundary data) \cite{Knudsen2009}. 

\subsection{The D-bar algorithm for ND maps}
The classical D-bar method \cite{Knudsen2006,Knudsen2009} is formulated for full-boundary Dirichlet-to-Neumann maps and is based on the results in \cite{Nachman1996,Novikov1988}. A first implementation was done in \cite{Siltanen2000} and further developments propsed by \cite{DeAngelo2010,Isaacson2004, Mueller2003}.
The algorithm we use is formulated for ND maps and works also for partial-boundary data \cite{Hauptmann2017}, it can be summarized by two essential steps. At first we compute an approximate scattering transform $\T^{\mathrm{ND}}(k)$ for $k\in\C$ from the ND map by evaluating the following integral 
\begin{equation}\label{eq:scatND}
\T^{\mathrm{ND}}(k) = \int_{\bndry} \left(\dnu e^{i\bar{k}\bar{\zeta}}\right)(-\NDdiff\dnu e^{ik\zeta})ds(\zeta).
\end{equation}
If the domain is chosen to be the unit disk, this can be simply evaluated as
\begin{align*}
\T^{\mathrm{ND}}(k) &= \int_{\bndry} i\bar{k}\bar{\zeta} e^{i\bar{k}\bar{\zeta}}(-\NDdiff\dnu ik\zeta e^{ik\zeta})ds(\zeta) \\
&= |k|^2\int_{\bndry}  \bar{\zeta} e^{i\bar{k}\bar{\zeta}}(\NDdiff) \zeta e^{ik\zeta}ds(\zeta).
\end{align*}

In practice we can not compute the scattering transform for all values of $k$. Furthermore, noise in the measurement data and inaccuracies in the model lead to a blow-up of the scattering data for large $|k|$ frequencies. Due to these restrictions it is common to restrict the computations to a disk of radius $R>0$. Additionally the partial-boundary introduces further instabilities to the data, hence we also enforce a cut-off $C_t$ for high amplitudes. Altogether we compute the scattering data as
\begin{equation}\label{eqn:scatOnR}
\T^{\mathrm{ND}}_{R}(k)=
\begin{cases}
\T^{\mathrm{ND}}(k) &\text{ if } 0<|k|< R\text{, and }\left|\text{Re}\left(\T^{\mathrm{ND}}\right)\right|,\;\left|\text{Im}\left(\T^{\mathrm{ND}}\right)\right|\leq C_t \\
0 &\mbox{ else. }
\end{cases}
\end{equation}

The second step is then to solve the D-bar equation with the obtained approximate scattering transform $\T^{\mathrm{ND}}_{R}$. This can be done independently for each $z\in\Omega$ by solving
\begin{equation}\label{eq:dbark_eq}
\dbar_k \mu_R(z,k) = \frac{1}{4\pi \bar{k}} \T^{\mathrm{ND}}_R(k)e_{-k}(z) \overline{\mu_R(z,k)},
\end{equation} 
with the unitary exponential $e_k(z)=e^{i(kz+\bar{k}\bar{z})}$.
In practice the D-bar equation \eqref{eq:dbark_eq} is solved via the integral equation
\begin{equation} \label{eq:dbark_eq_int}
\mu_R(z,k) = 1+\frac{1}{(2\pi)^2}\int_{|k|\leq R}\frac{\T^{\mathrm{ND}}_R(k^\prime)}{\overline{k^\prime}(k-k^\prime)}\;e_{-k^\prime}(z)\overline{\mu_R(z,k^\prime)}dk_1^\prime dk_2^\prime,
\end{equation}
and at the end the conductivity is computed by simply evaluating
\[
\sigma_R(z)=\mu_R(z,0)^2.
\]
The computation of $\sigma(z)$ for $z\in\Omega$ by \eqref{eq:dbark_eq_int} is independent of the mesh and parallelizable.

\subsection{Discussion of approximation error}
Let us shortly examine the approximation error of the two continuum models (CM and ECM) to electrode measurements from the CEM. By \cite[Theorem 4.1]{Hyvoenen2009}, we have that the approximation error of the CM to CEM measurements depends linearly on the largest diameter of the extended electrodes. By our result Theorem \ref{theo:approx}, the approximation error of the ECM to CEM measurments depends linearly on the length of the electrodes and not on the distance between the electrodes. In particular this means uncovered parts of the boundary do not influence the approximation error. 

\begin{table}[!h]
\centering 
\caption{Approximation error of ND maps from ECM and CM data compared to CEM measurements. The reference in the left column is done with 16 equally sized and spaced electrodes. For the following cases electrodes have been removed from the setup.}
\label{table:approxError}
\begin{tabular}{l|ccccccc}

$\ell^2$ errors & 16 elec. &  14 elec. & 12 elec. & 10 elec. & 8 elec. \\
\hline
CM data  & 0.0171  &  0.0601  &  0.0730   & 0.0770   & 0.0827 \\
\hline
ECM data  & 0.0084 &   0.0097 &   0.0077  &  0.0087  &  0.0156 \\
\hline
\end{tabular}
\end{table}

We aim to verify this error numerically. For that purpose we have simulated data with the complete electrode model, the continuum model, and the electrode continuum model. The simulations are all done on the same mesh and the resulting ND matrices are computed with 16 basis functions. For the CEM and ECM we have started with 16 equally sized and spaced electrodes covering the boundary of the unit disk. We have then removed 2 neighboring electrodes and kept the remaining electrodes at their initial position. We have repeated this until only 8 electrodes were left. The errors of the recorded ND matrices can be seen in Table \ref{table:approxError}. The difference between CM and CEM data gets gradually worse as the gap increases, whereas the error between ECM and CEM data stays rather constant. There is a peak for only 8 left electrodes, this might be due to a decrease of the constant in \eqref{eq:elecSizeCond} and hence an increase of the constant in the estimate of Theorem \ref{theo:approx}. It is notable that even for 16 electrodes the ECM data is closer to CEM data, compared to classical continuum data.

\subsection{Simulated data for trigonometric current patterns}\label{sec:compSim}
The first example is a phantom with two circles with conductivity 2, where one circle is positioned close to the missing boundary. We compare continuum full-boundary data to noisy measurements from the \emph{electrode continuum model} \eqref{eqn:ECM}. The basis functions are chosen as
\[
\varphi_n(\theta)=\left\{
\begin{array}{cl}
\frac{1}{\sqrt{\pi}}\sin(n\theta) &\text{ if } n<0,\\
\frac{1}{\sqrt{\pi}}\cos(n\theta) &\text{ if } n>0,
\end{array}\right.
\]
with $n\in\mathcal{N}=\{-8,\dots -1,1,\dots,8\}$. 
The geometry is chosen to be the unit disk and the data is simulated on a very fine uniform FEM mesh with $16384$ elements. Since the reconstruction algorithm is mesh independent we chose this fine mesh for all following reconstructions. For the partial boundary case we use 12 out of 16 electrodes, which are equally spaced with length $h=2\pi/32=0.1963$.
The partial-boundary current is produced by the nonorthogonal projection $Q$. The resulting partial ND matrix $\pNDdiffMat$ then has rank 11, whereas the ND matrix $\NDdiffMat$ has full rank 16. Additionally, we added 0.1\% of relative noise to each measured set of voltages. From the collected data, we then computed an approximation to the full ND matrix with the proposed algorithm in Section \ref{sec:Algo}. The resulting matrix has rank 15, due to the restriction of partial basis functions and 16 electrodes. The parameter in the optimization has been chosen as $\beta=0.06$, such that the error to the full ND map is minimized. The resulting reconstructions are shown in Figure \ref{fig:circPair_recons}, where the scattering transform is computed with $R=4$ and $C_t=25$.

\begin{figure}[t!]
\centering
\begin{picture}(350,300)
\put(71,0){\includegraphics[height=275 pt]{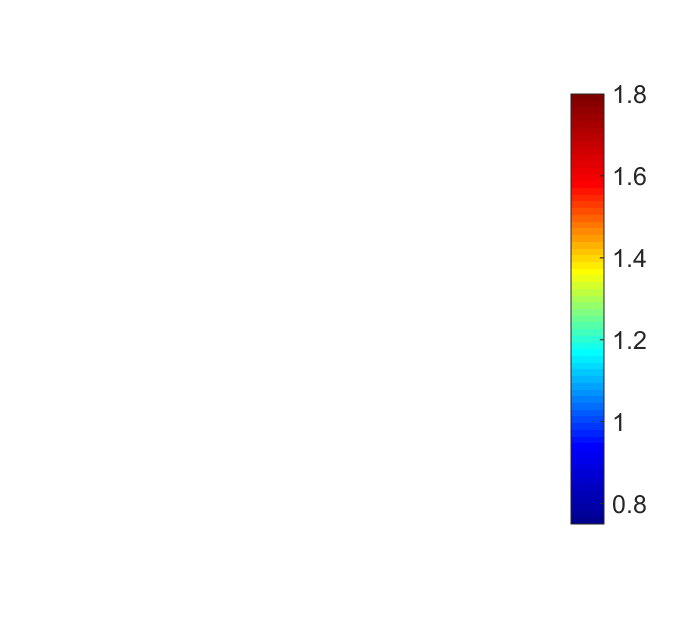}}
\put(-10,140){\includegraphics[width=175 pt]{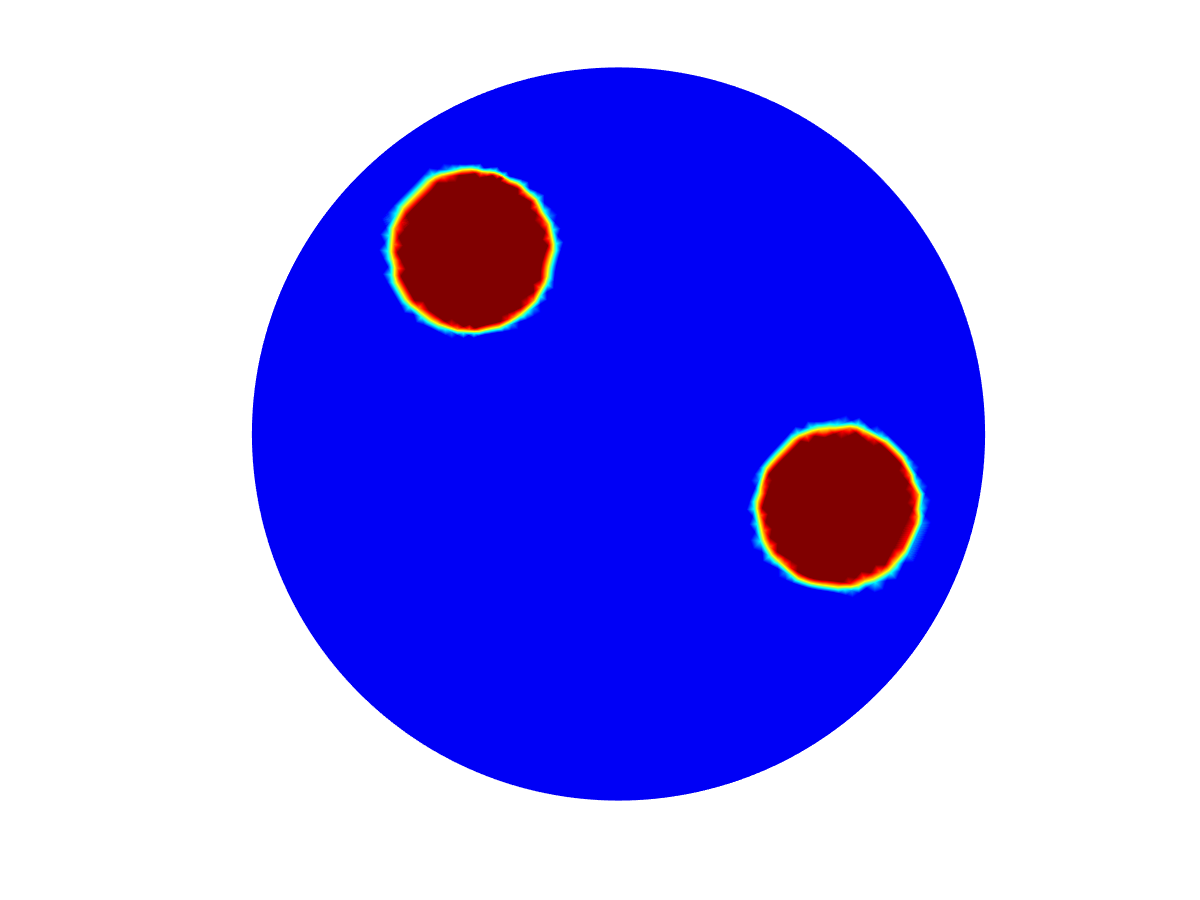}}
\put(140,140){\includegraphics[width=175 pt]{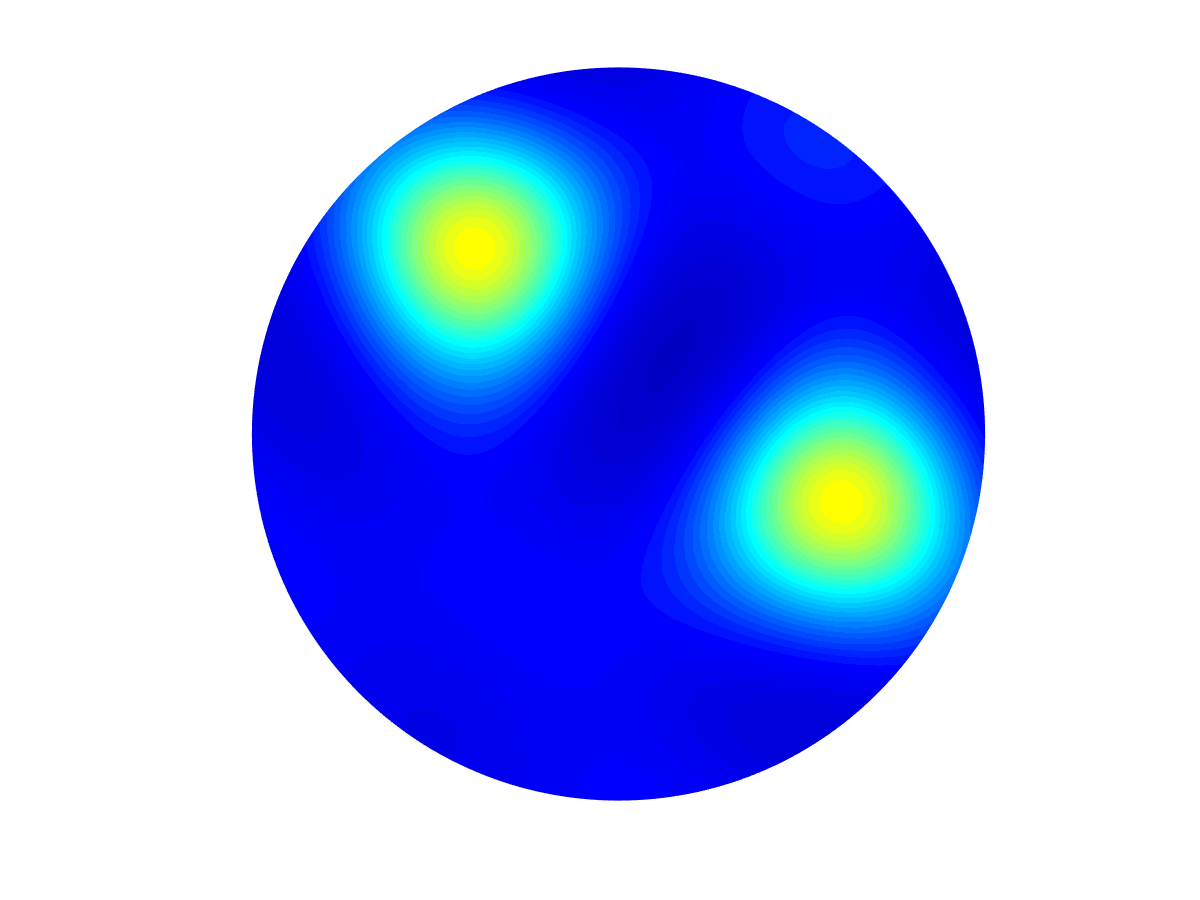}}
\put(-10,0){\includegraphics[width=175 pt]{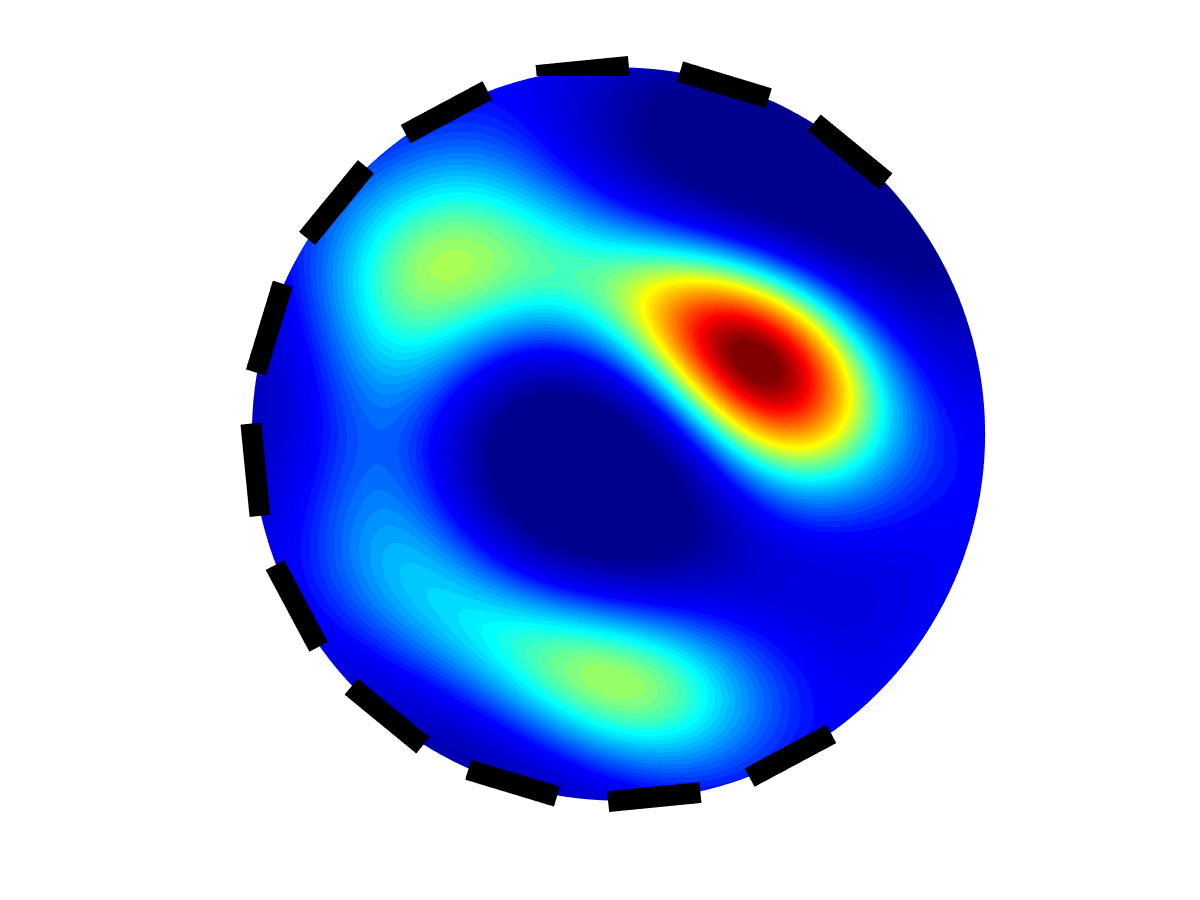}}
\put(140,0){\includegraphics[width=175 pt]{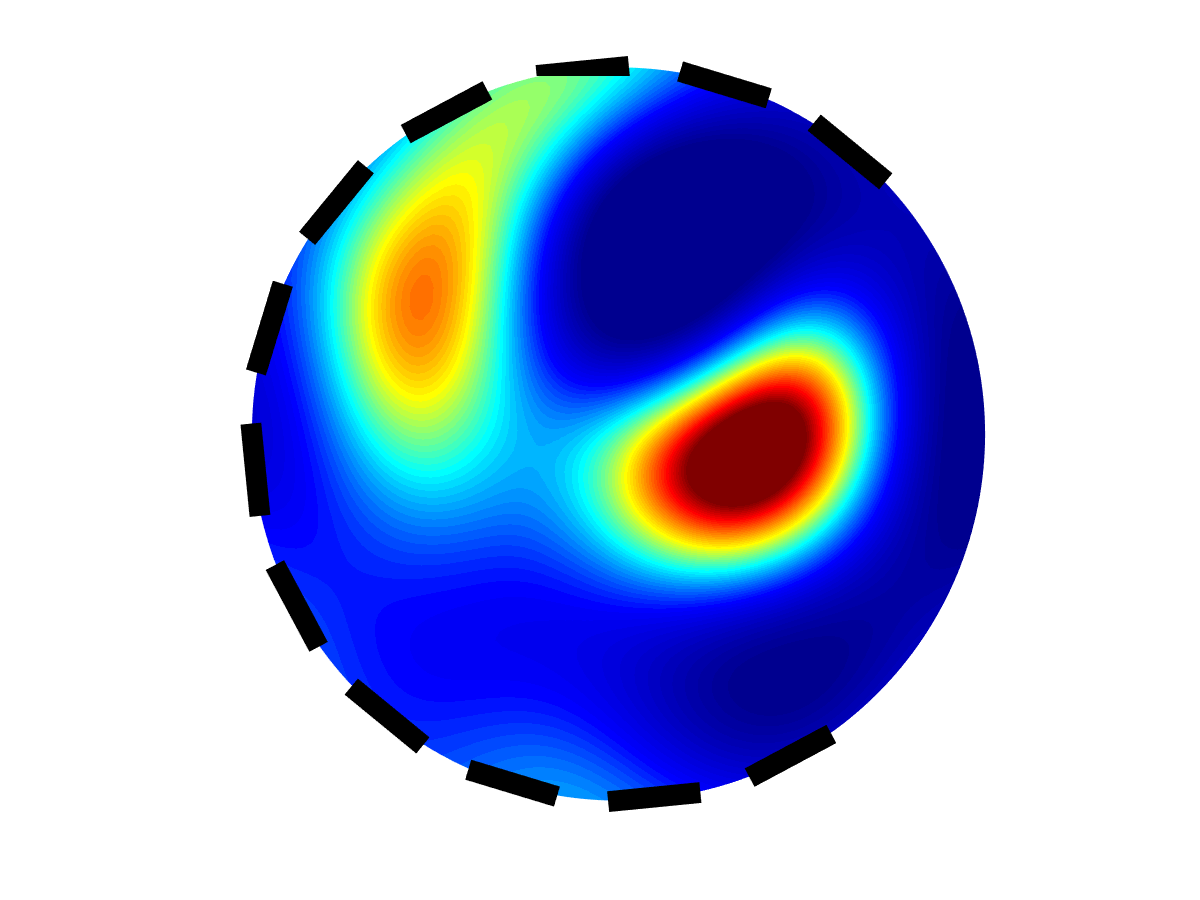}}
\put(38,135){\large{Partial-boundary}}
\put(45,125){\large{reconstruction}}
\put(180,135){\large{Approximated data}}
\put(195,125){\large{reconstruction}}
\put(55,270){\large{Phantom}}
\put(195,275){\large{Full-boundary}}
\put(195,265){\large{reconstruction}}

\end{picture}
\caption{\label{fig:circPair_recons} Reconstructions of the paired circle phantom in the continuum setting. The reference reconstruction (top right) has been computed from full-boundary continuum data without noise. The partial-boundary reconstruction (bottom left) is computed from ECM data with 12 electrodes on 75\% of the domain and additional 0.1\% relative noise. The improved reconstruction (bottom right) is obtained from the approximated ND matrix.}
\end{figure}
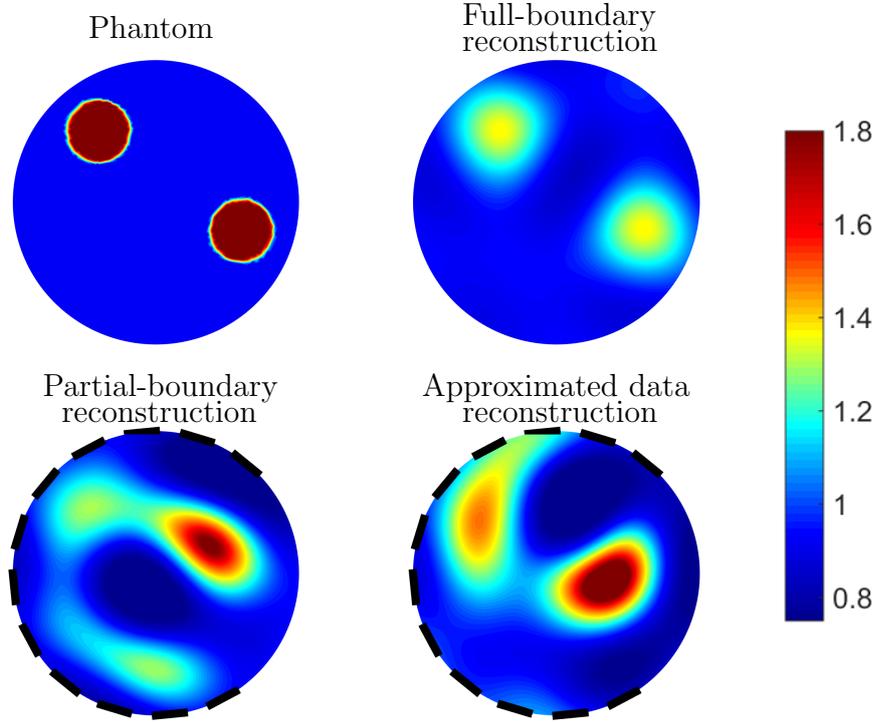 

The second test case uses a similar phantom with two circular inclusions with conductivity 2, but now we are using the \emph{complete electrode model} to produce more realistic data on 16 equally spaced electrodes with same size $h=2\pi/32$. The FEM mesh is finer close to the boundary and coarser in the inside of the domain, with a total of 1968 elements. The contact impedance is uniform on each electrode and set to $z=0.005$. The reference reconstruction uses all 16 electrodes with no noise, the partial electrode measurement uses 12 electrodes with 0.2\% additional relative noise. Following Theorem \ref{theo:approx} we interpret the measurement as from the ECM. By the proposed algorithm we compute an approximated ND matrix with $\beta=0.06$. The resulting reconstructions are shown in Figure \ref{fig:circPairCEM_recons}. The parameter for the scattering transform are $R=4$ and $C_t=25$.

\begin{figure}[t!]
\centering
\begin{picture}(350,300)
\put(71,0){\includegraphics[height=275 pt]{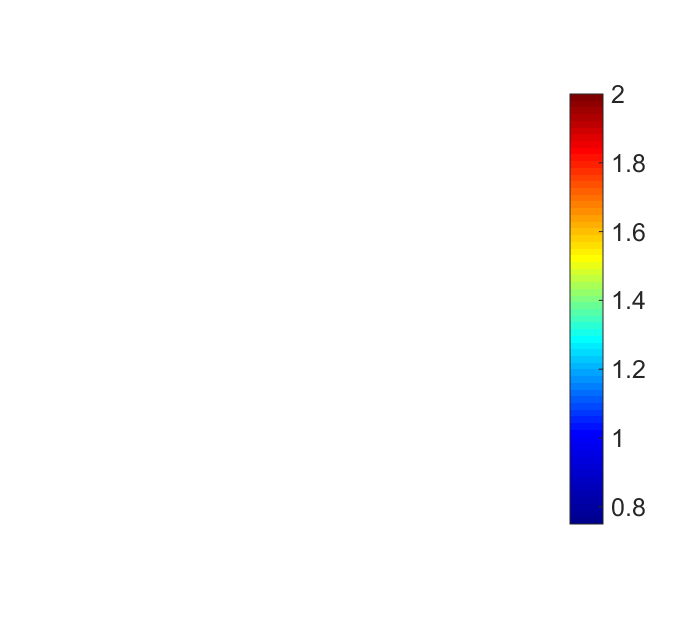}}
\put(-10,140){\includegraphics[width=175 pt]{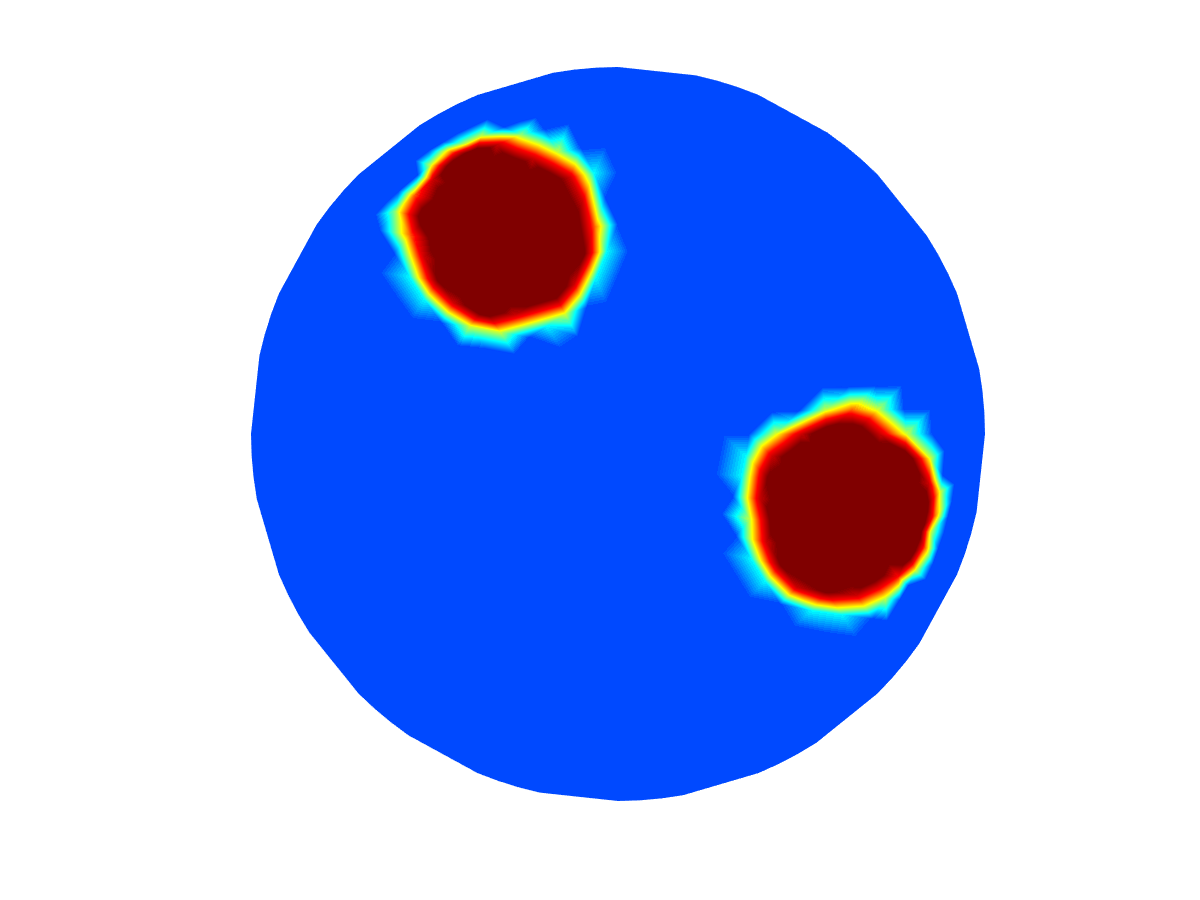}}
\put(140,140){\includegraphics[width=175 pt]{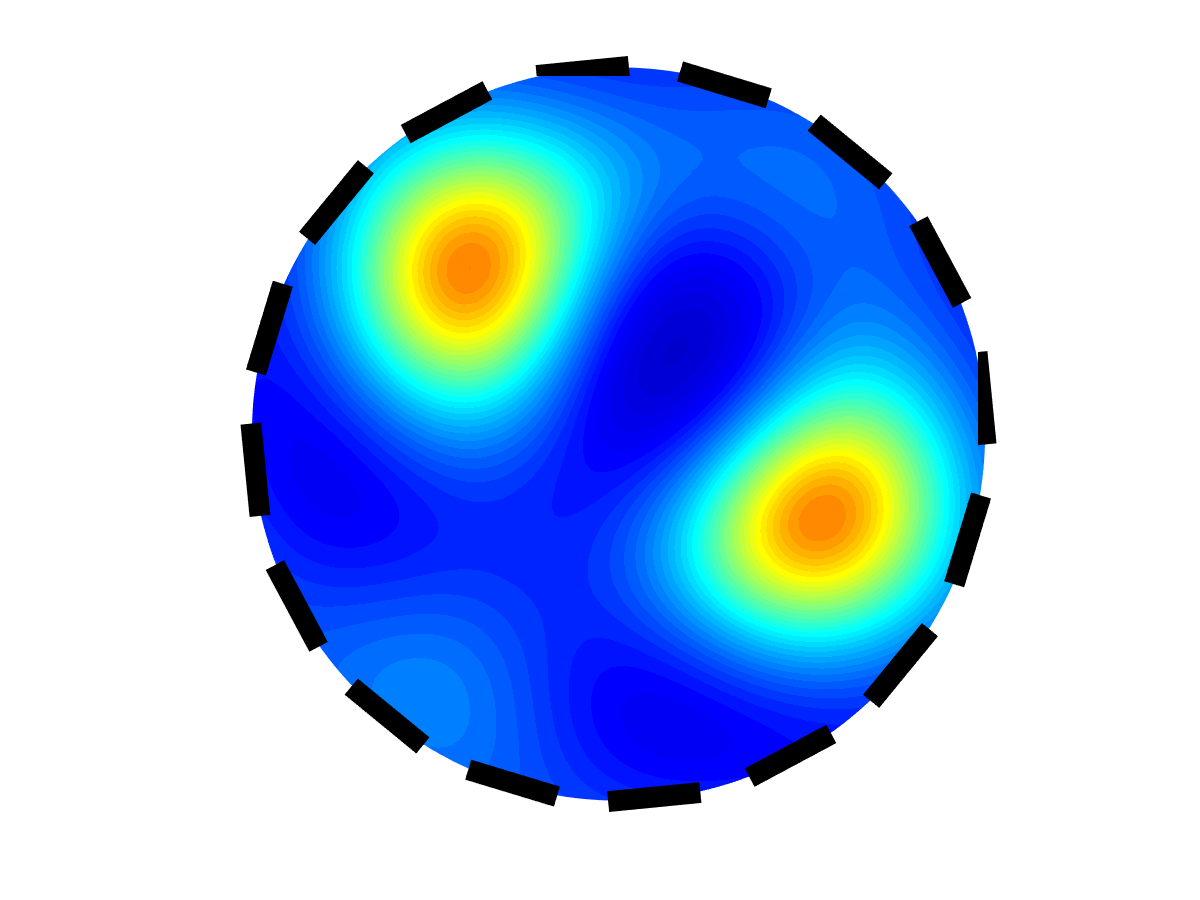}}
\put(-10,0){\includegraphics[width=175 pt]{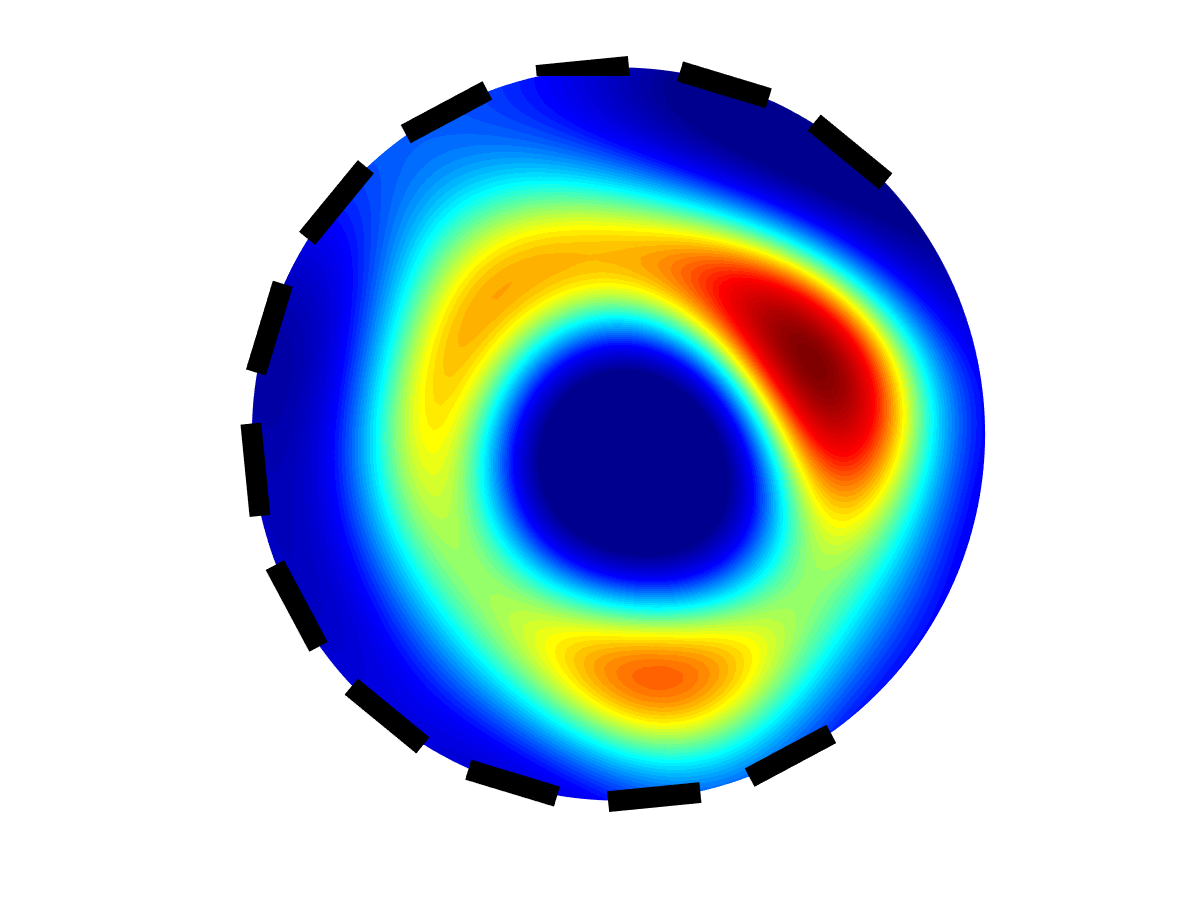}}
\put(140,0){\includegraphics[width=175 pt]{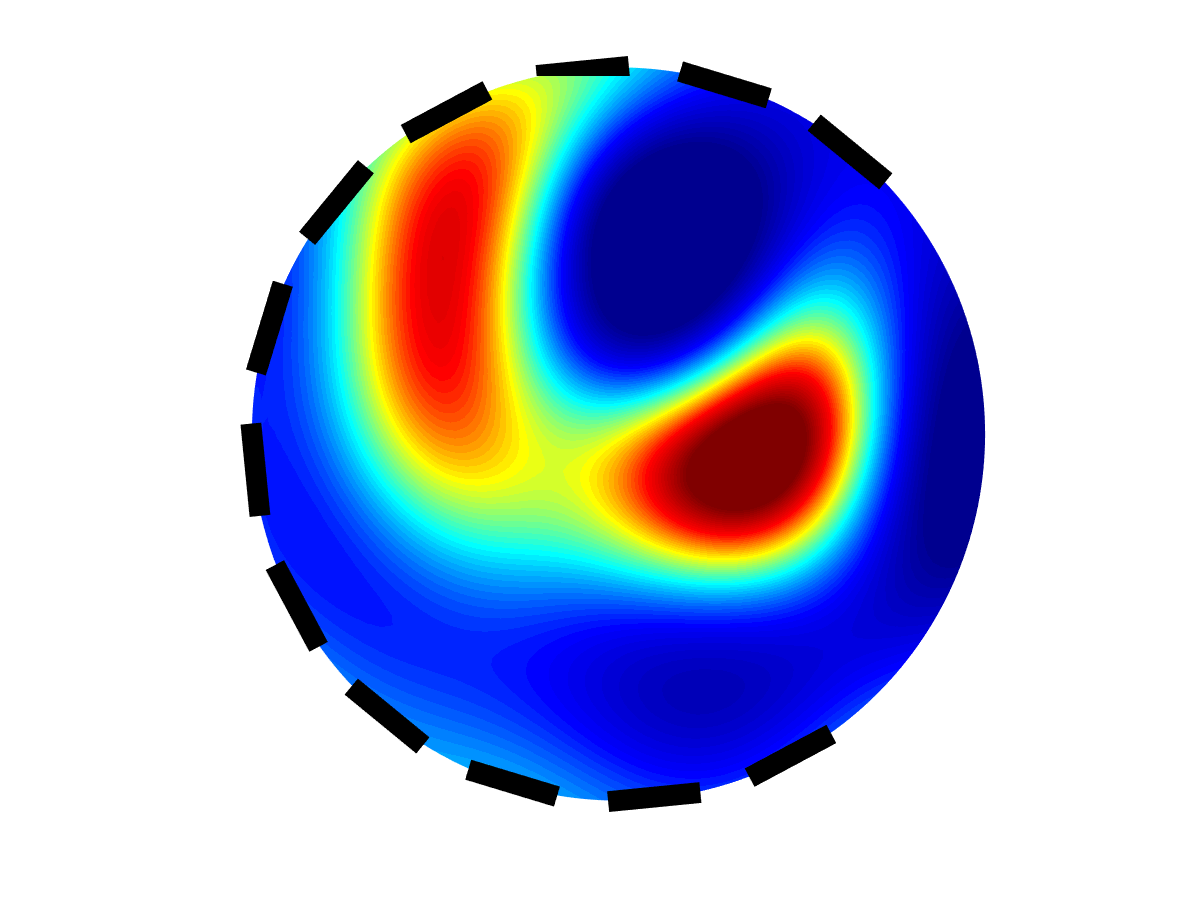}}
\put(38,135){\large{Partial-boundary}}
\put(45,125){\large{reconstruction}}
\put(180,135){\large{Approximated data}}
\put(195,125){\large{reconstruction}}
\put(55,270){\large{Phantom}}
\put(197,275){\large{Full-electrode}}
\put(195,265){\large{reconstruction}}

\end{picture}
\caption{\label{fig:circPairCEM_recons} Reconstructions of the paired circle phantom from CEM data. The reference reconstruction (top right) has been computed from 16 electrodes without additional noise. The partial-boundary reconstruction (bottom left) is computed from 12 electrodes on 75\% of the boundary and additional 0.2\% relative noise. The improved reconstruction (bottom right) is obtained from the approximated ND matrix}
\end{figure}
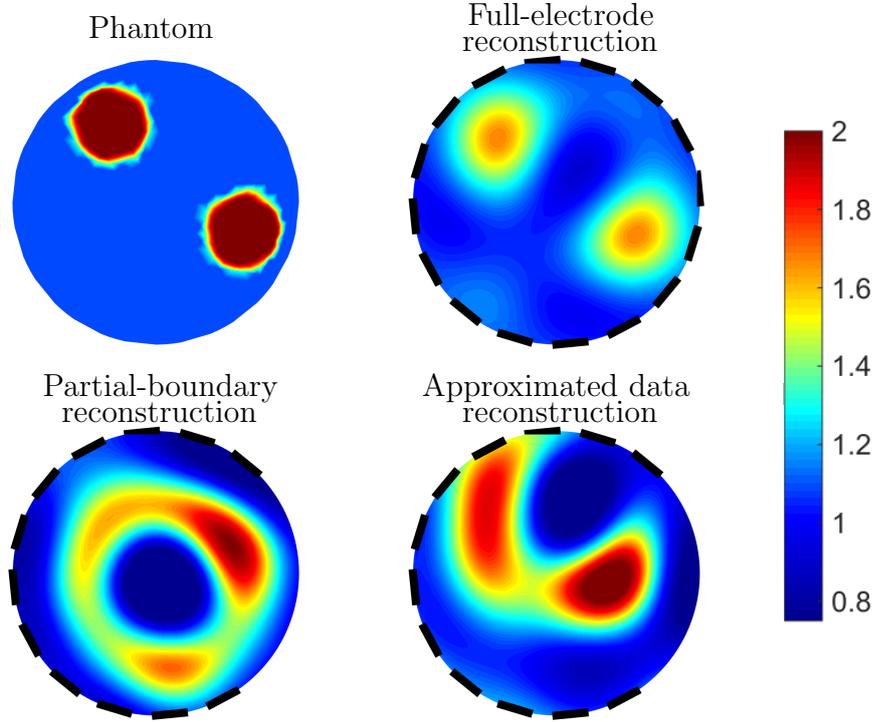

\subsection{Real measurement data for adjacent current pattern}
As final example we apply the methodology to real measurement data acquired with the KIT4 system in Kuopio at the University of Eastern Finland \cite{Hauptmann2017a,Kourunen2008}. The phantom consists of two high conductive metal rods (vertically translational ‎invariant) inserted to the tank filled with normal tap water. The tank has a diameter of 28 cm, the 16 metallic electrodes are 2.5 cm wide, 7 cm high, and equally spaced.
The background data has been taken with only water in the tank. 

The biggest difficulty for the proposed algorithm is that the system uses pairwise injection current patterns, which are not orthonormal and hence we first need to transform the measurement (by change of basis) to the trigonometric basis used to represent the ND maps. The used data was acquired with an adjacent current pattern from 16 electrodes. The partial data is taken with 10 active electrodes, which leaves only 9 of the original adjacent current patterns for the computations. The contact impedance is roughly $1.3\cdot 10^{-5}\Omega \text{cm}$.

Since the exact partial-boundary map is not known for the transformed adjacent current pattern, we computed the corresponding basis functions numerically from the data of a constant background. This can be done by transforming the obtained background measurement to the trigonometric basis and then normalizing to 1. Using this basis in the approximation algorithm with $\beta=0.5$, we were able to compute an approximated ND matrix from the partial data. The resulting reconstructions are displayed in Figure \ref{fig:circPairKIT_recons}. The parameter used to compute the scattering transform are $R=4$ and $C_t=25$.

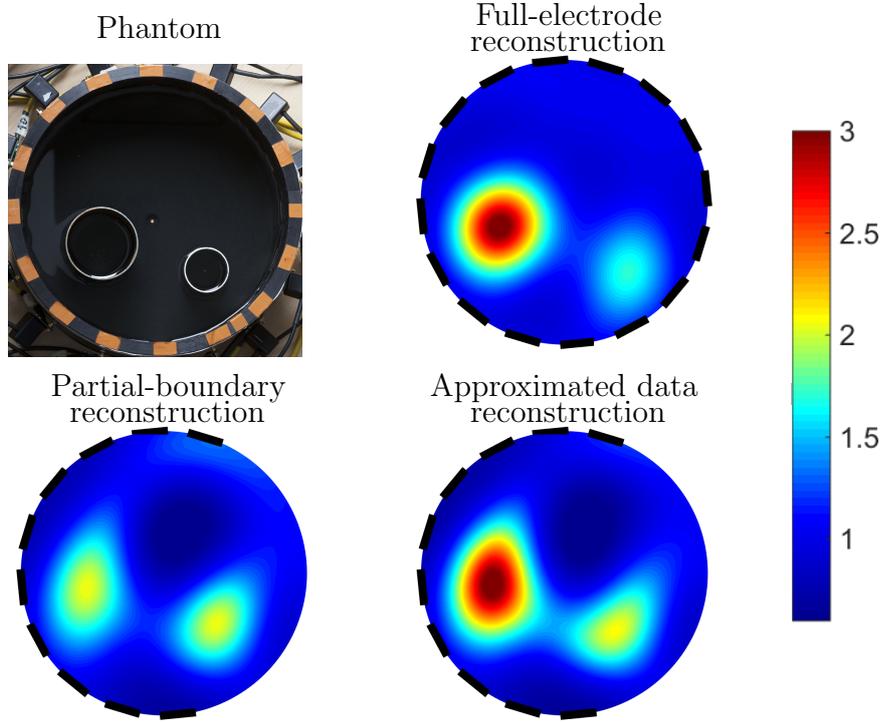
\begin{figure}[t!]
\centering
\begin{picture}(350,300)
\put(71,0){\includegraphics[height=275 pt]{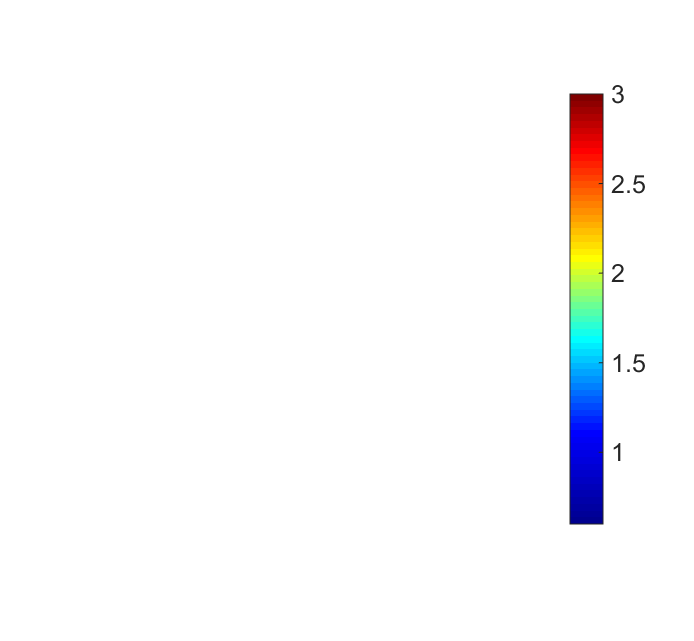}}
\put(22,150){\includegraphics[width=110 pt]{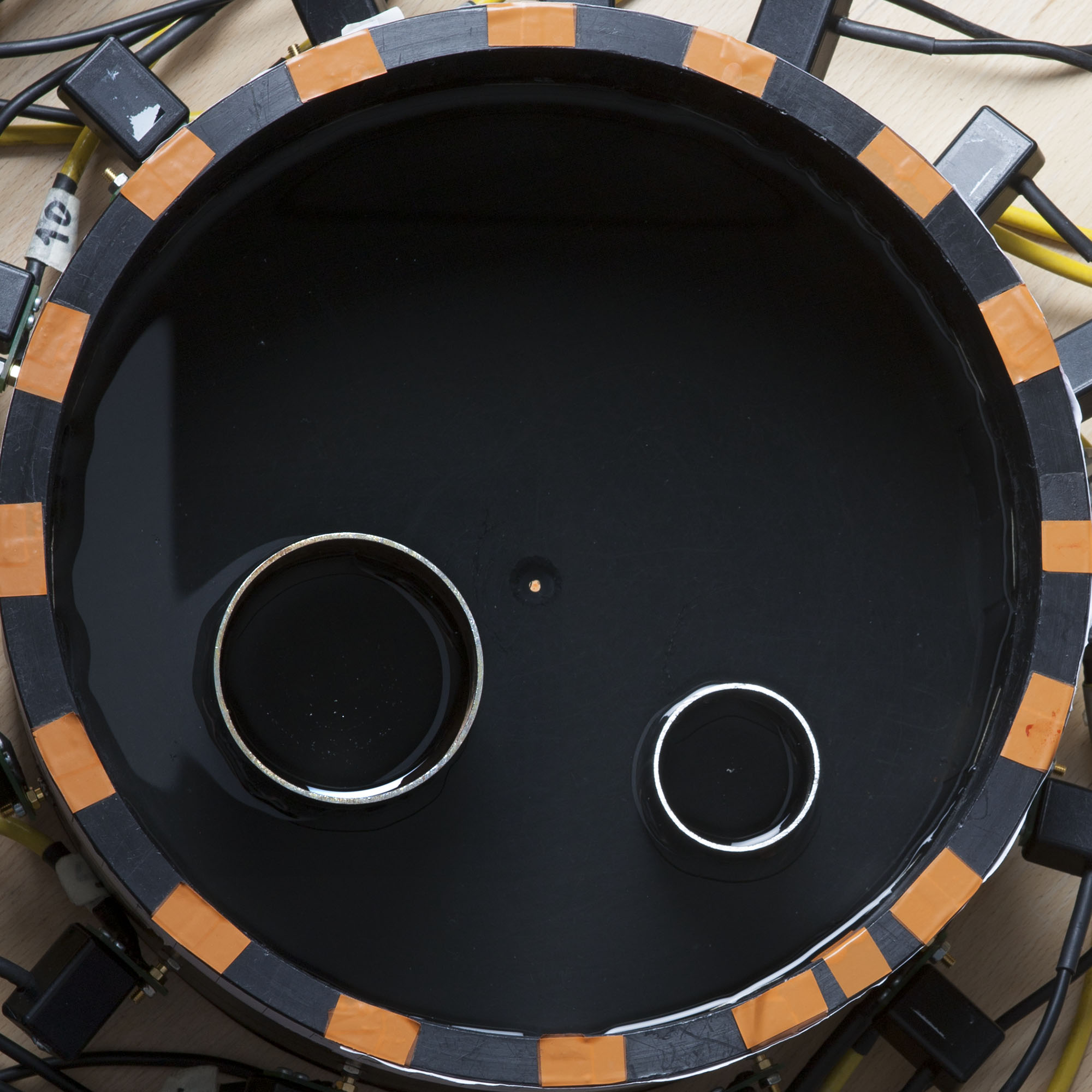}}
\put(140,140){\includegraphics[width=175 pt]{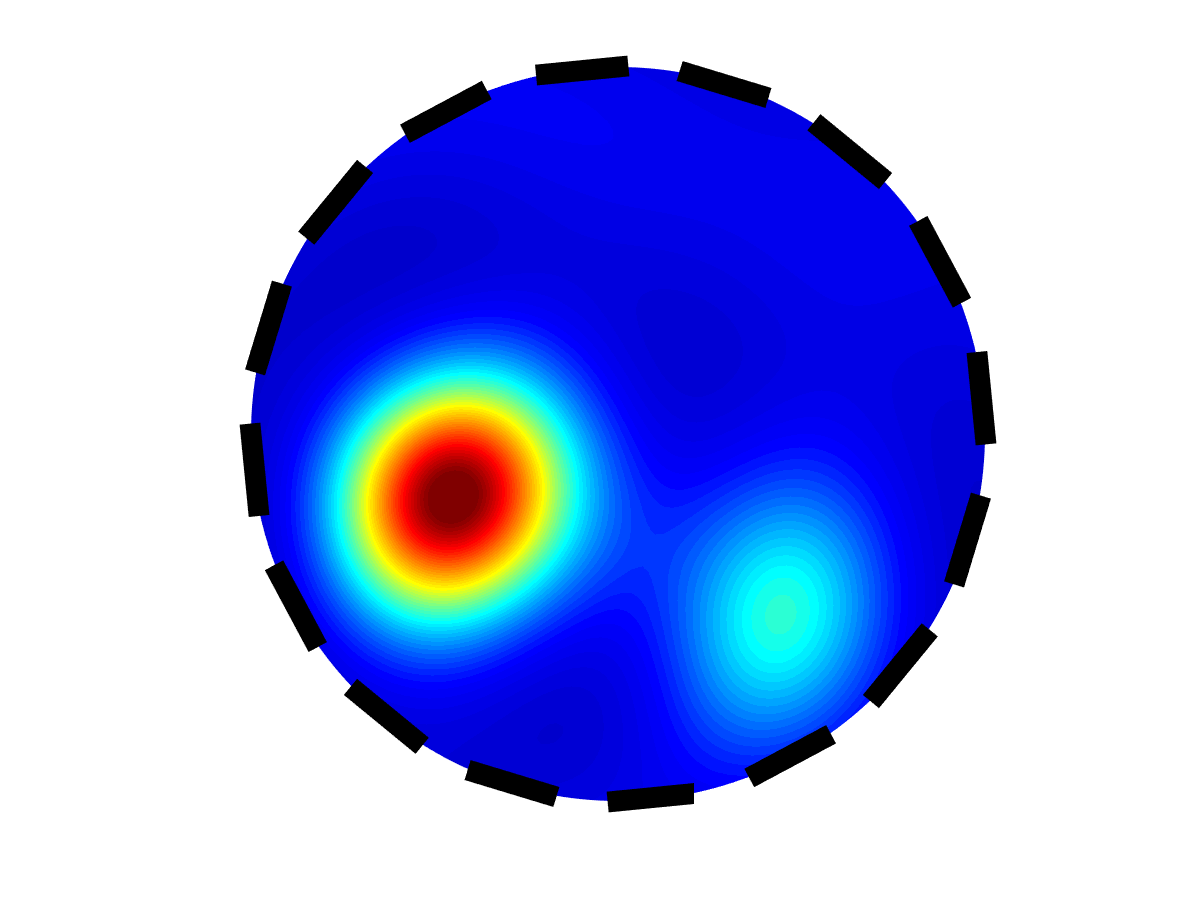}}
\put(-10,0){\includegraphics[width=175 pt]{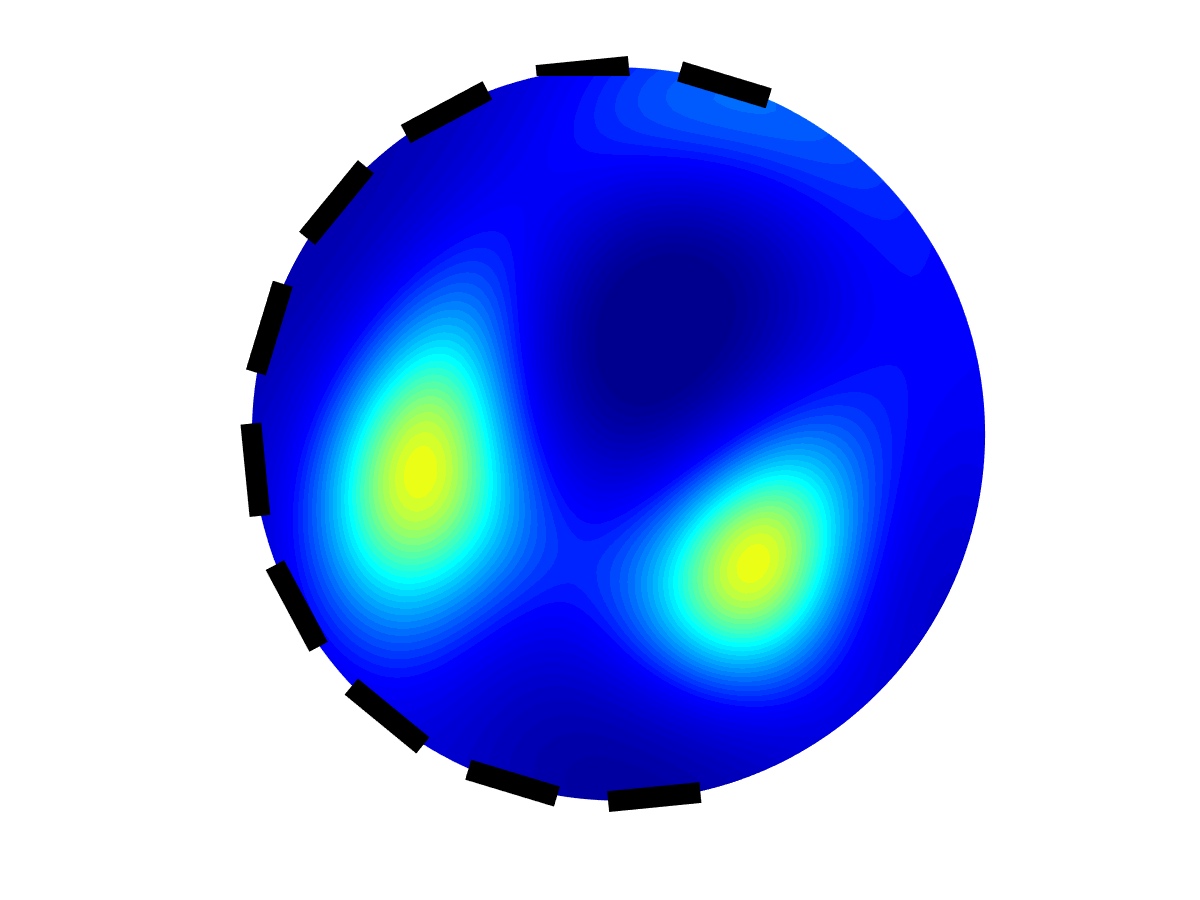}}
\put(140,0){\includegraphics[width=175 pt]{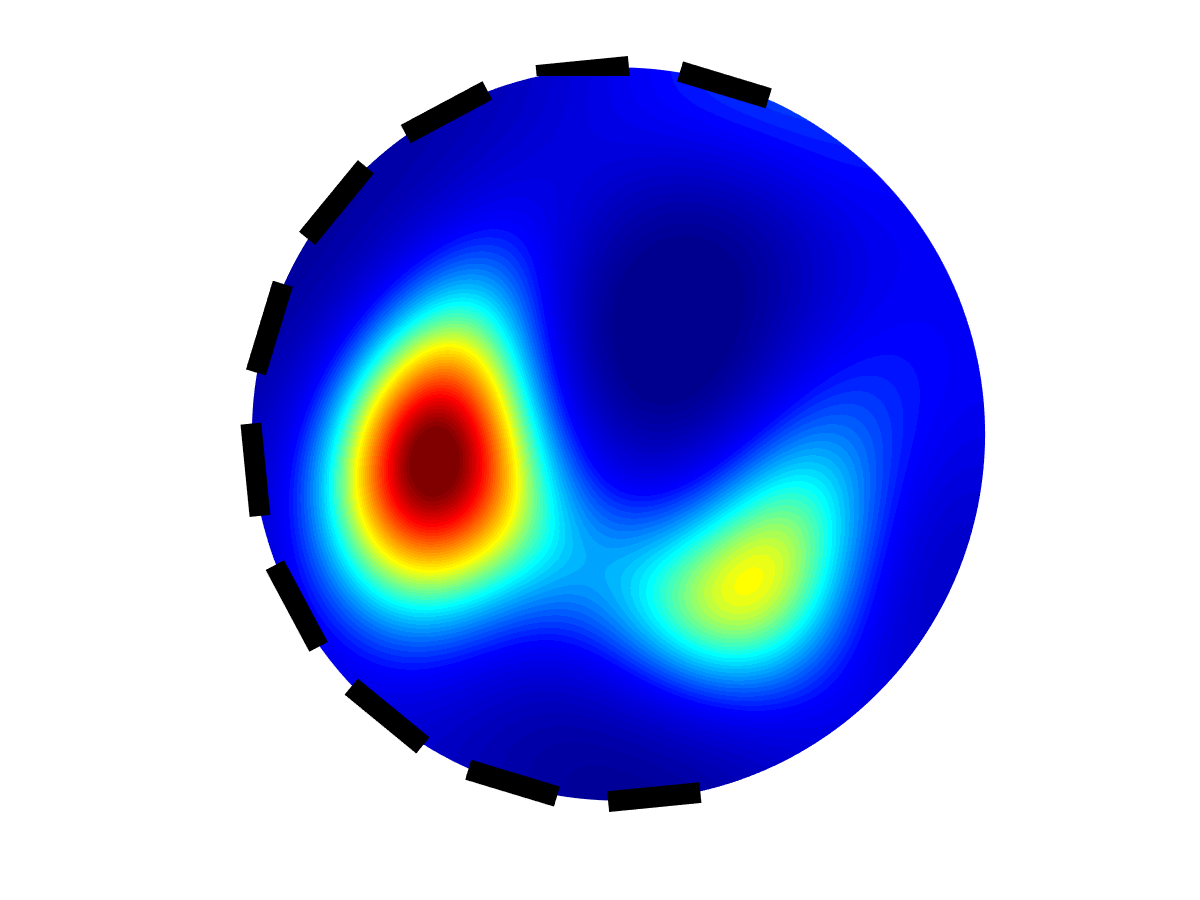}}
\put(38,135){\large{Partial-boundary}}
\put(45,125){\large{reconstruction}}
\put(180,135){\large{Approximated data}}
\put(195,125){\large{reconstruction}}
\put(55,270){\large{Phantom}}
\put(197,275){\large{Full-electrode}}
\put(195,265){\large{reconstruction}}

\end{picture}
\caption{\label{fig:circPairKIT_recons} Reconstructions of a paired circle phantom from real measurement data with the KIT4 system and an adjacent current pattern. The reference reconstruction (top right) has been computed from all 16 electrodes. The partial-boundary reconstruction (bottom left) is computed from 10 electrodes on 62.5\% of the boundary. The improved reconstruction (bottom right) is obtained from the approximated ND matrix}
\end{figure}

\subsection{Discussion of computational results}\label{sec:discus}
The reconstructions illustrate that the introduced methodology is capable of recovering certain information contained in the full ND maps from measurements conducted on electrodes only covering a part of the boundary. Whereas we are not able to perfectly recover the ND map or a reconstruction of the conductivity, we can certainly improve the reconstruction quality and information content of the image.
Especially the second example from CEM measurements illustrates nicely how noise corrupted data can ruin the reconstructions, but with the proposed algorithm we are able to clean the data and recover the two important circular inclusions. It should be noted here that the full-boundary reconstructions for both simulated cases do not include additional measurement noise. The improvement of data can also be seen in Table \ref{table:reconError}, where we present the relative error to full-boundary data and the corresponding reconstructions. The error in ND maps clearly gets better, whereas the reconstruction error just slightly improves even though the visual effects are far more striking. We emphasize here that EIT is a highly nonlinear inverse problem.

In case of real measurement data, the partial data reconstruction is already quite good, but the conductivity values are incorrect. Here the approximation procedure corrects the conductivity values and cleans some boundary artifacts. The effectiveness is also illustrated by quite an improvement in the reconstruction error, as shown in Table \ref{table:reconError}.

\begin{table}[!h]
\centering 
\caption{Evaluation of relative $\ell^2$ errors to full-boundary measurements. In case of CEM and real measurement data the full-boundary data corresponds to data with all electrodes active (and no noise).}
\label{table:reconError}
\begin{tabular}{l|ccccccc}

rel. $\ell^2$ errors & part. ND map&  approx. ND map   & part. recon. & approx. recon\\
\hline
ECM data  & 68.21\%& 58.12\% & 18.49\%  & 17.73\%  \\
\hline
CEM data& 66.72\% & 46.30\% & 25.06\% & 27.92\% \\
\hline
KIT4 data& 53.02\% & 43.66\% & 24.23\% &   19.55\% \\ 
\hline
\end{tabular}
\end{table}

An important feature of the proposed algorithm is the improvement of linear independent basis functions representing the ND maps. The effectiveness of D-bar methods also depends on the amount of basis functions by which we can expand the auxiliary functions in the integral \eqref{eq:scatND}. In case of real data we were able to improve the rank of the partial ND map from 9 to 15 for the approximated ND map and this alone can lead to an improvement in the reconstruction procedure.

The computation times of the proposed algorithm are well below a second, since after the data is collected the whole algorithm consists of solving 2 linear systems. 

\section{Conclusions}\label{sec:conclusion}
Data collected only on a part of the boundary is inherently different to the idealized full-boundary continuum case. This is especially relevant for real measurements, since any electrode model can be considered as a partial-boundary problem. 
Therefore, we have introduced in this study a methodology that relates partial-boundary data to full-boundary data and we proposed an algorithm that is capable of computing an approximate ND map from noisy electrode measurements that cover only parts of the boundary. For this purpose we have introduced a framework to model electrode measurements in a continuum setting more accurately and have shown in Theorem \ref{theo:approx} that we can model measurements from the \emph{complete electrode model} up to an error that depends linearly on the size of electrodes, but is independent of the missing boundary.

We have then formulated an algorithm that computes an approximation to the (full-boundary) ND matrix. The approximation process consists of solving a minimization problem to obtain the coefficients of the ND matrix and the solution can be simply computed by solving a regularized normal equation. The whole process takes less than a second and hence can be considered as a reasonable precomputing step. We have then demonstrated the effectiveness of the methodology for noisy simulated data and real measurement data. Reconstructions were computed by a D-bar algorithm for ND maps, but we would like to point out that the presented improvements from partial-boundary data have a potential impact on all reconstruction algorithms that use ND maps as data input.

More work has to be done for pairwise injection current patterns, for which the representation of the partial-boundary map is not known. Furthermore, one has more freedom of choosing pairwise injection current patterns close to the missing boundary, that includes current patterns that jump over the missing region. Finally we will concentrate further research on improvements of the approximation functional \eqref{eqn:oneStepOpt} and a thorough numerical study of the approximation properties.

\section*{Acknowledgments}
This work was supported by the Academy of Finland through the Finnish
Centre of Excellence in Inverse Problems Research 20122017, decision number
250215. The author was partially supported by FiDiPro project of the Academy of Finland, decision number 263235.

Personally, I would like to thank Ville Kolehmainen, Tuomo Savolainen, and Aku Sepp\"anen from the Inverse Problems research group at University of Finland, for the possibility to use the KIT4 measurement data. I also thank Xiaoqun Zhang for the helpful discussions that led to the start of this project during my visit at Shanghai Jiao Tong University in 2015. Finally, I thank Samuli Siltanen for his continuous support and guidance.

I also thank Nuutti Hyv\"onen and the anonymous referees for their valuable comments for improvements.
\bibliographystyle{siam}
\bibliography{Inverse_problems_references_2017}

\end{document}